\documentclass[openacc,a4paper]{rsproca_new}

\usepackage{amsfonts}
\usepackage{graphicx}
\usepackage{epstopdf}
\usepackage{algorithm}
\ifpdf
\DeclareGraphicsExtensions{.eps,.pdf,.png,.jpg}
\else
\DeclareGraphicsExtensions{.eps}
\fi
\usepackage{bm}
\usepackage{multirow}
\usepackage{algpseudocode}
\usepackage[caption=false,font=footnotesize,indention=0cm]{subfig}
\usepackage{array}
\usepackage{cleveref}
\usepackage{xcolor}

\usepackage{enumitem}
\setlist[enumerate]{leftmargin=.5in}
\setlist[itemize]{leftmargin=.5in}

\def\R{\mathbb{R}}
\def\N{\mathbb{N}}

\def\first#1{{\setlength{\fboxsep}{2pt}\fcolorbox{black}{black!15}{\textbf{#1}}}}
\def\second#1{{\setlength{\fboxsep}{2pt}\fcolorbox{black!50}{black!5}{#1}}}

\def\aiuse#1{{\vskip5.5pt\noindent \textcolor{jobcolor}{\fontsize{9}{11}\selectfont Declaration of AI use.}\fontsize{8}{11}\selectfont\enskip #1}}

\makeatletter
\newcommand\footnoteref[1]{\protected@xdef\@thefnmark{\ref{#1}}\@footnotemark}
\makeatother

\newtheorem{theorem}{\bf Theorem}[section]
\newtheorem{lemma}{\bf Lemma}[section]

\newtheorem{remark}{\bf Remark}[section]

\titlehead{Research}

\begin{document}
\newpage
\title{A nonlinear spectral core-periphery detection method for multiplex networks}

\author{
Kai Bergermann$^{1}$, Martin Stoll$^{1}$ and Francesco Tudisco$^{2,3}$}

\address{$^{1}$Department of Mathematics, Chemnitz University of Technology, 09107 Chemnitz, Germany\\
$^{2}$School of Mathematics, The University of Edinburgh, Edinburgh EH93FD, UK\\
$^{3}$School of Mathematics, Gran Sasso Science Institute, 67100 L'Aquila, Italy}

\subject{
applied mathematics, computational mathematics, computational physics
}


\corres{Francesco Tudisco\\
\email{f.tudisco@ed.ac.uk}}

\begin{abstract}
Core-periphery detection aims to separate the nodes of a complex network into two subsets: a core that is densely connected to the entire network and a periphery that is densely connected to the core but sparsely connected internally.
The definition of core-periphery structure in multiplex networks that record different types of interactions between the same set of nodes on different layers is nontrivial since a node may belong to the core in some layers and to the periphery in others.
We propose a nonlinear spectral method for multiplex networks that simultaneously optimises a node and a layer coreness vector by maximising a suitable nonconvex homogeneous objective function by a provably convergent alternating fixed point iteration.
We derive a quantitative measure for the quality of a given multiplex core-periphery structure that allows the determination of the optimal core size.
Numerical experiments on synthetic and real-world networks illustrate that our approach is robust against noisy layers and significantly outperforms baseline methods while improving the latter with our novel optimised layer coreness weights.
As the runtime of our method depends linearly on the number of edges of the network it is scalable to large-scale multiplex networks.
\end{abstract}

\maketitle


\section{Introduction}\label{sec:introduction}

Complex networked systems are all around us: interactions of users, items, stops, genes, proteins, or machines can all be recorded by pairwise connections \cite{newman2003structure,estrada2012structure}.
In recent years, multilayer networks in which different types of interactions are represented in different layers have gained increased attention \cite{kivela2014multilayer,boccaletti2014structure}.
Consequently, several quantitative means for the structural analysis of complex networks such as community detection \cite{mucha2010community,mercado2018power,bergermann2021semi} or centrality analysis \cite{de2015ranking,tudisco2018node,bergermann2022fast} have been generalised from single- to multi-layered networks.
In this work, we focus on node-aligned multiplex networks without inter-layer edges.

Core-periphery structure is a mesoscopic network property combining the concepts of community structure and centrality as it seeks to partition a network's node set into a set of core and a set of periphery nodes.
Core nodes are characterised by a high degree of connectivity to the full network while periphery nodes may be strongly connected to the core but are sparsely connected to other peripheral nodes.
Common synonyms are rich-club \cite{colizza2006detecting} and core-fringe structure \cite{benson2019link}.
The seminal work by Borgatti and Everett formalised an ideal core-periphery structure as an L-shaped block adjacency matrix in which core nodes are connected to all other nodes but no edges are present between pairs of periphery nodes \cite{borgatti2000models}.
This binary partitioning problem of nodes into two sets, however, leads to combinatorial optimisation problems that are computationally prohibitive to solve exactly even for relatively small networks.
Instead, a number of approaches for the approximate solution exist in the literature, some of which heavily rely on centrality measures, cf.\ e.g., \cite{borgatti2000models,holme2005core,boyd2010computing,kitsak2010identification,brusco2011exact,csermely2013structure,rombach2014core,zhang2015identification,cucuringu2016detection,lu2016h,jia2019random,fasino2020fast}.
Existing approaches for core-periphery detection in single-layer networks typically rely on maximising an objective function
\begin{equation}\label{eq:objective_general_kernel}
f(\bm{x}) = \sum_{i,j=1}^n \bm{A}_{ij} \kappa(\bm{x}_i,\bm{x}_j)
\end{equation}
for a node coreness vector $\bm{x}\in\R^n_{\geq 0}$, with large entries for nodes belonging to the core, subject to suitable constraints on $\bm{x}$ preventing $f(\bm{x})$ from blowing up.
Here, $\bm{A}_{ij} \geq 0$ denotes the edge weight between nodes $i$ and $j$ with $i,j=1,\dots,n$ and $\kappa: \R_{\geq 0} \times \R_{\geq 0} \rightarrow \R_{\geq 0}$ a kernel function.
A commonly chosen kernel function is the maximum kernel $\kappa(\bm{x}_i,\bm{x}_j) = \max\{\bm{x}_i,\bm{x}_j\}$, which relates to the ideal L-shape model in the sense that $\bm{A}_{ij}$ contributes to \cref{eq:objective_general_kernel} if at least one of the nodes $i$ or $j$ belongs to the core.

More recently, the logistic core-periphery random graph model has been introduced as a relaxed version of the ideal L-shape model of core-periphery structure \cite{tudisco2019nonlinear,jia2019random}.
It motivates the usage of a smoothed maximum function
\begin{equation}\label{eq:smoothed_max}
\kappa(\bm{x}_i,\bm{x}_j) = \left( \bm{x}_i^\alpha + \bm{x}_j^\alpha \right)^{1/\alpha} \underset{\alpha \rightarrow \infty}{\longrightarrow} \max \{ \bm{x}_i, \bm{x}_j \}
\end{equation}
as kernel function for vectors $\bm{x}\in\R^n_{> 0}$ \cite{tudisco2019nonlinear}.
Recent advances in nonlinear Perron--Frobenius theory guarantee the unique solvability of the non-convex constrained optimisation problem of maximising \cref{eq:objective_general_kernel} subject to norm constraints on $\bm{x}$ by means of a nonlinear power iteration \cite{tudisco2018node,tudisco2019nonlinear,gautier2019perron}.

The formulation of core-periphery structure in multilayer networks is still in its infancy.
In multiplex networks, each entity is represented by one node in each layer and the definition of multiplex coreness is non-trivial since, in general, each layer may possess a different core-periphery structure or none at all.
The most prominent approach to date utilises linear combinations
\begin{equation}\label{eq:coreness_linear_combination}
\bm{x} = \bm{c}_1 \bm{d}^{(1)} + \cdots + \bm{c}_L \bm{d}^{(L)}
\end{equation}
of the node degree vectors $\bm{d}^{(1)}, \dots , \bm{d}^{(L)}\in\R^n_{\geq 0}$ of the $L$ layers of the multiplex network weighted by a layer coreness vector $\bm{c}\in\R^L_{\geq 0}$ with $\|\bm{c}\|_1=1$, which needs to be chosen a-priori \cite{battiston2018multiplex}.
In the absence of expert knowledge suggesting otherwise, the authors propose the choice of equal layer weights, i.e., layer coreness vectors $\bm{c}=\frac{1}{L}\bm{1}$ with $\bm{1}\in\R^L$ the vector of all ones and results of experiments for other choices of $\bm{c}$ are presented.
Note that the framework proposed in \cite{battiston2018multiplex} also allows non-linear combinations of the degree vectors more general than \cref{eq:coreness_linear_combination}.
Furthermore, a recent extension of \cref{eq:coreness_linear_combination} additionally takes inter-layer connections, i.e., edges across layers into account and the influence of different layer weight vectors $\bm{c}$ is studied experimentally \cite{le2021multi}.

In this work, we formulate an objective function $f_\alpha(\bm{x},\bm{c})$ for core-periphery detection in multiplex networks that simultaneously optimises a node and a layer coreness vector $\bm{x}\in\R^n_{>0}$ and $\bm{c}\in\R^L_{>0}$, respectively.
It is based on the smoothed maximum kernel function \cref{eq:smoothed_max} and can be seen as a generalisation of the nonlinear spectral core-periphery detection method for single-layer networks \cite{tudisco2019nonlinear}.
We show that under suitable conditions, the unique maximum of $f_\alpha(\bm{x},\bm{c})$ can be computed by an alternating nonlinear power iteration summarised in \Cref{alg}.
Furthermore, we prove convergence of $f_\alpha(\bm{x},\bm{c})$ to a local maximum in cases in which the assumptions for unique solvability are not satisfied.
We present different quantitative indicators for the quality of a given multiplex core-periphery partition and derive optimal core sizes.

We perform extensive numerical experiments on two types of multiplex networks:
single-layer networks with strong core-periphery structure with a second random noise layer as well as real-world multiplex networks with an unknown distribution of core-periphery structure across the layers.
The results indicate that our proposed method outperforms the multilayer degree method \cite{battiston2018multiplex} in all but one experiment while improving the latter with the novel optimised layer coreness weights.
Additionally, we distinctly outperform existing single-layer methods applied to aggregated versions of the multiplex networks.
Experiments on the first type of networks illustrate that our approach allows the automated detection of noise layers, which is widely acknowledged to be a difficult yet highly relevant problem in multilayer network science \cite{venturini2023icml,peel2022statistical,choobdar2019assessment,mercado2018power}. 
As the runtime of our method depends linearly on the number of edges in the network, \Cref{alg} is scalable to large-scale multiplex networks.

The remainder of this paper is structured as follows.
\Cref{sec:problem_formulation} formulates the multiplex core-periphery detection problem.
We propose our nonlinear spectral method for multiplex networks in \Cref{sec:MP_NSM} and prove convergence guarantees to global and local optima in different parameter settings.
In \Cref{sec:algorithm}, we summarise our algorithm before \Cref{sec:measuring_coreness} discusses the determination of optimal core sizes and the quality of multiplex core-periphery partitions.
Finally, \Cref{sec:numerical_experiments} reports numerical experiments on artificial and real-world multiplex networks.

\section{Problem formulation}\label{sec:problem_formulation}

We consider connected, binary, and undirected multiplex networks without self-edges and inter-layer edges and choose a third-order adjacency tensor representation
\begin{equation*}
\bm{A} \in \mathbb{R}^{n \times n \times L}, \quad \text{with} \quad \left[ \bm{A}_{ij}^{(k)} \right]_{i,j=1, \dots ,n}^{k = 1 , \dots , L},
\end{equation*}
where $n\in\mathbb{N}$ denotes the number of nodes and $L\in\mathbb{N}$ the number of layers, i.e., the $n \times n$ adjacency matrix of layer $k$ is found in the $k$-th frontal slice of $\bm{A}$.
Hence, we have
\begin{equation*}
\bm{A}_{ij}^{(k)} = \begin{cases} 1 & \text{if nodes }i \text{ and } j \text{ with } i\neq j \text{ are connected on layer } k, \\ 0 & \text{otherwise,}\end{cases}
\end{equation*}
as well as $\bm{A}_{ij}^{(k)} = \bm{A}_{ji}^{(k)}$ for $i,j=1, \dots , n$ and $k=1, \dots , L$.

We seek a partition of the network's nodes into a set of core and a set of periphery nodes.
However, each node is represented in every layer and may belong to the core in some layers and to the periphery in others.
To combine the coreness information from different layers by a layer coreness vector $\bm{c}\in\R^L_{> 0}$, we adapt the approach from \cite{battiston2018multiplex} of considering linear combinations \cref{eq:coreness_linear_combination}.
In the multilayer degree approach of \cite{battiston2018multiplex}, the coreness vector of the nodes is computed as $\bm{c}_1 \bm{d}^{(1)} + \cdots + \bm{c}_L \bm{d}^{(L)}$, where $ \bm{d}^{(k)}$ is the degree vector of the layer $k$,  and  $\bm{c}$ is a layer-weighting vector that must be chosen a-priori. Note that, due to the linearity of the degree, this vector coincides with the degree vector of the aggregated multiplex network with adjacency matrix $\sum_{k=1}^L \bm{c}_k \bm{A}_{ij}^{(k)}$. Leveraging this second interpretation, we propose a multiplex core-periphery objective function that simultaneously optimises the node score and the layers' weights $\bm c$. 

Due to its success in the single-layer case, we adopt the logistic core-periphery objective function \cite{tudisco2019nonlinear,jia2019random} and propose an extension  of the corresponding nonlinear spectral method.
Our multiplex objective function reads
\begin{equation}\label{eq:objective}
f_\alpha (\bm{x},\bm{c}) = \sum_{k=1}^L \sum_{i,j=1}^n \bm{c}_k \bm{A}_{ij}^{(k)} (\bm{x}_i^\alpha + \bm{x}_j^\alpha)^{1/\alpha},
\end{equation}
which we want to maximise simultaneously with respect to $\bm{x}\in\R^n_{> 0}$, the (common) node coreness vector across the layers, and $\bm{c}\in\R^L_{> 0}$, the layer coreness weights vector. Due to \cref{eq:smoothed_max}, when $\alpha$ is large enough, $f_\alpha$ takes large values each time that an edge exists $\bm{A}_{ij}^{(k)}\neq 0$ on a layer with large core score $\bm{c}_k$, at least one of the two nodes $i,j$ is in the core, i.e.\ either $\bm{x}_i$ or $\bm{x}_j$ is large.

\begin{sloppypar}
	To avoid trivial solutions (with all the entries blowing up to infinity), we require norm constraints on the vectors $\bm{x}$ and $\bm{c}$. This is not a limitation as we  are ultimately seeking a relative score for both nodes and layers. 
	Thus, we restrict $\bm{x}$ to \mbox{$\mathcal{S}_p^+=\{\bm{x}\in\R^n_{>0} \colon \|\bm{x}\|_p=1 \}$} and $\bm{c}$ to \mbox{$\mathcal{S}_q^+=\{\bm{c}\in\R^L_{>0} \colon \|\bm{c}\|_q=1 \}$}. We will discuss the roles of $p,q>1$ in \Cref{sec:MP_NSM}.
	Overall, our nonlinear spectral method for core-periphery detection in multiplex networks relies on the solution of the problem
\end{sloppypar}
\begin{equation}\label{eq:constrained_optimization_problem}
\max_{\bm{x}\in\mathcal{S}_p^+,\bm{c}\in\mathcal{S}_q^+}~f_\alpha (\bm{x},\bm{c}).
\end{equation}

\begin{sloppypar}
	Note  that $f_\alpha(\bm{x},\bm{c})$ is 1-homogeneous in both $\bm{x}$ and $\bm{c}$ independently, i.e., we have $f_\alpha(\beta\bm{x},\bm{c})=f_\alpha(\bm{x},\beta\bm{c})=\beta f_\alpha(\bm{x},\bm{c})$ for all $\beta\in\R$, which we denote by $f_\alpha(\bm{x},\cdot)\in\text{hom}(1)$ and $f_\alpha(\cdot,\bm{c})\in\text{hom}(1)$.
	This shows that the restriction of $\bm{x}$ to $\mathcal{S}_p^+$ and of $\bm{c}$ to $\mathcal{S}_q^+$ can be generalised to vectors of arbitrary norms.
\end{sloppypar}

By \cite[Lemma 4.2]{tudisco2019nonlinear} and the 1-homogeneity of $f_\alpha(\bm{x},\bm{c})$ in both $\bm{x}$ and $\bm{c}$, we can reformulate \cref{eq:constrained_optimization_problem} as the equivalent unconstrained optimisation problem
\begin{equation}\label{eq:unconstrained_optimization_problem}
\max_{\bm{x}\in\R^n_{>0},\bm{c}\in\R^L_{>0}}~g_\alpha(\bm{x},\bm{c}) \qquad \text{with} \qquad g_\alpha(\bm{x},\bm{c}) = \frac{f_\alpha (\bm{x},\bm{c})}{\|\bm{x}\|_p \|\bm{c}\|_q}\, .
\end{equation}
Note that for the special case of one layer, i.e., $L=1$ and $\bm{c}=1$, we recover the logistic core-periphery model for single-layer networks as proposed in \cite{tudisco2019nonlinear}.

\section{Nonlinear spectral method for multiplex networks}\label{sec:MP_NSM}

In this section, we first observe that the critical point equation for the optimisation problem \cref{eq:unconstrained_optimization_problem} corresponds to a generalized nonlinear eigenvalue problem.
Thus, we introduce a nonlinear spectral method for multiplex networks that addresses the solution of \cref{eq:unconstrained_optimization_problem}.
For a specific range of the parameters $\alpha$, $p$, and $q$, we prove that the method converges to the global solution for any starting point.
For the remaining set of parameters, we show monotonic ascent and convergence to a local maximum.
The method is detailed in \Cref{alg}.
Numerical experiments, however, indicate that solutions with relatively small parameters $p$ and $q$ outside the range of global convergence often yield better core-periphery partitions for single-layer and multiplex networks in terms of different metrics, which are introduced in \Cref{sec:measuring_coreness}.

Throughout this paper, we denote the gradient of the $p$-norm by $J_p(\bm{x}) := \nabla \|\bm{x}\|_p$.
For $\bm{x}\in\R^n_{>0}$ and elementwise powers its expression reads $J_p(\bm{x}) = \|\bm{x}\|_p^{1-p}\bm{x}^{p-1}$.
Additionally, we denote the H\"older conjugate of $p>1$ by  $p^\ast$, which is defined via $1/p+1/p^\ast=1$.
Analogously, we write $J_q(\bm{c}) := \nabla \|\bm{c}\|_q$ and define $q^\ast$ via $1/q+1/q^\ast=1$.
With this notation we have $J_p^{-1}=J_{p^\ast}$ on $\mathcal{S}_p^+$ and $J_q^{-1}=J_{q^\ast}$ on $\mathcal{S}_q^+$, since $\|\cdot\|_p$ and $\|\cdot\|_q$ are G\^ateaux differentiable, see e.g.\ \cite{gautier2021global}.

Consider the unconstrained optimisation problem \cref{eq:unconstrained_optimization_problem}.
Setting the gradient of $g_\alpha(\bm{x},\bm{c})$ with respect to $\bm{x}$ to zero yields
\begin{align}
\nabla_{\bm{x}} g_\alpha(\bm{x},\bm{c}) & = \frac{1}{\|\bm{x}\|_p^2 \|\bm{c}\|_q^2} \left\{ \nabla_{\bm{x}} f_\alpha(\bm{x},\bm{c}) \|\bm{x}\|_p \|\bm{c}\|_q - f_\alpha(\bm{x},\bm{c}) \|\bm{c}\|_q J_p(\bm{x}) \right\}\nonumber\\
& = \frac{1}{\|\bm{x}\|_p} \left\{ \frac{\nabla_{\bm{x}} f_\alpha(\bm{x},\bm{c})}{\|\bm{c}\|_q} -g_\alpha(\bm{x},\bm{c}) J_p(\bm{x}) \right\} \overset{!}{=} \bm{0}\nonumber\\
& \Leftrightarrow \frac{\nabla_{\bm{x}} f_\alpha(\bm{x},\bm{c})}{\|\bm{c}\|_q} = g_\alpha(\bm{x},\bm{c}) J_p(\bm{x})\label{eq:nonlinear_generalized_ev_problem}\\
& \Leftrightarrow \bm{x} = J_{p^\ast} \left( \frac{\nabla_{\bm{x}} f_\alpha(\bm{x},\bm{c})}{\|\bm{c}\|_q} \right) = J_{p^\ast} (\nabla_{\bm{x}} f_\alpha(\bm{x},\bm{c})),\nonumber
\end{align}
where the last identity is obtained by applying $J_{p^\ast}$ on both sides of the second equality, and where we  use the fact that $J_{p^\ast}\in\text{hom}(0)$ is scale-invariant, which can be verified with a straightforward computation. 
Notice that from \cref{eq:nonlinear_generalized_ev_problem} it follows that any critical point $(\bm{x}^\ast,\bm{c}^\ast)$ of $g_\alpha(\bm{x},\bm{c})$ with respect to the variable $\bm{x}$ solves a generalized eigenvalue equation with eigenvalue given by $g_\alpha(\bm{x},\bm{c})$ evaluated at $(\bm{x}^\ast,\bm{c}^\ast)$.
An analogous computation for the gradient of $g_\alpha(\bm{x},\bm{c})$ with respect to $\bm{c}$ yields
\begin{equation*}
\bm{c} = J_{q^\ast} (\nabla_{\bm{c}} f_\alpha(\bm{x},\bm{c})),
\end{equation*}
which shows that the same eigenvector interpretation holds also for $\bm{c}$.

The last two expressions have the form of two fixed point equations, which lead to the iterative updating rules
\begin{equation}\label{eq:iteration_multilayer}
\bm{x}_{k+1} = J_{p^\ast} (\nabla_{\bm{x}} f_\alpha(\bm{x}_k,\bm{c}_k)), \quad \bm{c}_{k+1} = J_{q^\ast} (\nabla_{\bm{c}} f_\alpha(\bm{x}_k,\bm{c}_k))\, .
\end{equation}
A direct computation shows that \cref{eq:iteration_multilayer} coincides with the iteration steps indicated in \Cref{alg}.
The following theorem provides conditions on the choice of $p,q$ and $\alpha$ such that \cref{eq:iteration_multilayer}, and thus \Cref{alg}, converges to the global maximum of \cref{eq:unconstrained_optimization_problem}.
Consider the following $2\times 2$ matrix 
\[\bm M = 
\begin{bmatrix} \frac{2|\alpha-1|}{p-1} & \frac{1}{p-1}\\[.7em]
\frac{2}{q-1} & 0
\end{bmatrix}\, .
\]
The following main result holds. For brevity, we moved the detailed proof to \Cref{sec:appendix_proof_unique_solution}.

\begin{theorem}\label{thm:unique_solution}
	For  parameters $\alpha,p,q>1$ for which
	\begin{equation*}
	\rho(\bm{M}) < 1\, ,
	\end{equation*}
	where  $\rho$ denotes the spectral radius, the iteration \cref{eq:iteration_multilayer} converges to the global maximum $(\bm{x}^\ast,\bm{c}^\ast)$ of \cref{eq:unconstrained_optimization_problem} for any starting vectors $\bm{x}_0\in\R_{>0}^n$ and $\bm{c}_0\in\R_{>0}^L$ with linear rate of convergence, i.e., \mbox{$\|\bm{x}_k-\bm{x}^\ast\| + \|\bm{c}_k-\bm{c}^\ast\| \leq C \rho(\bm{M})^k \left( \|\bm{x}_0-\bm{x}^\ast\| + \|\bm{c}_0-\bm{c}^\ast\| \right)$} for some constant $C\in\R_{>0}$.
\end{theorem}

\begin{remark}
	By \Cref{thm:unique_solution}, the alternating fixed point iteration underlying \cref{eq:iteration_multilayer} and \Cref{alg} is equivalent to the power iteration of a generalized nonlinear eigenvalue problem.
	To see this, consider \cref{eq:nonlinear_generalized_ev_problem} and its counterpart that is obtained when computing $\nabla_{\bm{c}} g_\alpha(\bm{x},\bm{c}) = \bm{0}$.
	Combining both in vector-valued form yields
	\begin{equation*}
	 \begin{bmatrix}\frac{\nabla_{\bm{x}} f_\alpha(\bm{x},\bm{c})}{\|\bm{c}\|_q}\\\frac{\nabla_{\bm{c}} f_\alpha(\bm{x},\bm{c})}{\|\bm{x}\|_p}\end{bmatrix} = g_\alpha(\bm{x},\bm{c}) \begin{bmatrix}J_p(\bm{x})\\J_q(\bm{c})\end{bmatrix}.
	\end{equation*}
	Here, both vectors represent nonlinear operators acting on $\begin{bmatrix}\bm{x}\\\bm{c}\end{bmatrix}\in\R_{>0}^{n+L}$ while the scalar quantity $g_\alpha(\bm{x},\bm{c})$ represents the corresponding positive eigenvalue to be maximized.
\end{remark}

If $p,q$, and $\alpha$ do not satisfy the assumption of \Cref{thm:unique_solution}, we cannot guarantee convergence to the unique global maximum of \cref{eq:unconstrained_optimization_problem}.
However, \Cref{alg} still converges to a local maximum. The next main theorem shows this result by proving that the objective function $g_\alpha(\bm{x},\bm{c})$ is strictly increasing at each iteration unless we have reached a critical point. We defer the proof to \Cref{sec:appendix_proof_local_convergence}.

\begin{theorem}\label{thm:local_optimum}
	Given the iteration \eqref{eq:iteration_multilayer}, we either have $g_\alpha(\bm{x}_{k+1},\bm{c}_{k+1}) > g_\alpha(\bm{x}_k,\bm{c}_k)$ or the iteration reaches a critical point, i.e., $[\bm{x}_k,\bm{c}_k]^T=[\bm{x}_{k+1},\bm{c}_{k+1}]^T$ and $\nabla_{\bm{x}} g_\alpha(\bm{x}_k, \bm{c}_k) = \nabla_{\bm{c}} g_\alpha(\bm{x}_k, \bm{c}_k) = \bm{0}$.
\end{theorem}

For the special case $L=1$ and $\bm{c}=1$, \Cref{thm:local_optimum} represents a novel result for the single-layer nonlinear spectral method \cite{tudisco2019nonlinear} that often gives rise to more L-shaped reordered adjacency matrices in comparison to the globally optimal solutions obtained in the setting of \Cref{thm:unique_solution}.

\section{Algorithm}\label{sec:algorithm}

\begin{algorithm}[t]
	\vspace{0.5em}
	\begin{tabular}{lll}
		Input:
		& $\bm{A}\in\R^{n \times n \times L}_{\geq 0},$ & Adjacency tensor.\\\vspace{1mm}
		& $\bm{x}_0\in\R^n_{>0},$ & Initial node coreness vector.\\
		& $\bm{c}_0\in\R^L_{>0},$ & Initial layer coreness vector.\\
	\end{tabular}\\
	\begin{tabular}{ll}
		Parameters: & $\alpha,p,q>1, \mathrm{tol}\in\R_{>0}, \mathrm{maxIter}\in\N.$\\
	\end{tabular}\vspace{1mm}
	
	\begin{algorithmic}[1]
		\State $\bm{x}_0 =\bm{x}_0/\| \bm{x}_0 \|_{\frac{p}{p-1}}$
		\State $\bm{c}_0 = \bm{c}_0/\| \bm{c}_0 \|_{\frac{q}{q-1}}$
		\For{$i=1:\mathrm{maxIter}$}
		\vspace{1mm}
		\State $\widetilde{\bm{x}} = \nabla_{\bm{x}} f_\alpha(\bm{x}_0, \bm{c}_0)$
		\vspace{1mm}
		\State $\bm{x} = \left( \widetilde{\bm{x}}/\| \widetilde{\bm{x}} \|_{\frac{p}{p-1}} \right)^{\frac{1}{p-1}}$
		\vspace{1mm}
		\State $\widetilde{\bm{c}} = \nabla_{\bm{c}} f_\alpha(\bm{x}_0, \bm{c}_0)$
		\vspace{1mm}
		\State $\bm{c} = \left( \widetilde{\bm{c}}/\| \widetilde{\bm{c}} \|_{\frac{q}{q-1}} \right)^{\frac{1}{q-1}}$
		\vspace{1mm}
		\If{$\| \bm{x} - \bm{x}_0 \| < \mathrm{tol}$ \textbf{and} $\| \bm{c} - \bm{c}_0 \| < \mathrm{tol}$}
		\State \textbf{break}
		\Else
		\State $\bm{x}_0 = \bm{x}$
		\State $\bm{c}_0 = \bm{c}$
		\EndIf
		\EndFor
	\end{algorithmic}
	\vspace{1mm}
	\begin{tabular}{lll}
		\vspace{1mm}
		Output: & $\bm{x}\in\mathcal{S}_p^+,$ & Optimised node coreness vector.\\
		& $\bm{c}\in\mathcal{S}_q^+,$ & Optimised layer coreness vector.
	\end{tabular}
	\caption{Nonlinear spectral core-periphery detection method for multiplex networks.}\label{alg}
\end{algorithm}

The iteration \cref{eq:iteration_multilayer} derived in the previous section directly leads to the alternating optimisation scheme summarised in \Cref{alg}.

The expressions of the gradients of $f_\alpha(\bm{x},\bm{c})$ with respect to $\bm{x}$ and $\bm{c}$ presented in the proof of \Cref{lemma:coefficient_matrix} suggest an efficient implementation to leverage the sparsity structure of $\bm{A}$.
To this end, we implement \Cref{alg} with a customised data type for sparse third-order tensors $\bm{A}\in\R^{n \times n \times L}$ in \texttt{julia} \cite{bezanson2017julia}.
This leads to a linear computational complexity in the number of edges in the multiplex network, which depends approximately linearly on $nL$ with $n$ the number of nodes and $L$ the number of layers for many real-world networks with power-law degree distributions.
The \texttt{julia} code implementing \Cref{alg} is publicly available under \url{https://github.com/COMPiLELab/MPNSM}.

The choice of the parameters $p$ and $q$ qualitatively influences the obtained core-periphery detection results.
Relatively large values of $q$ generally lead to layer coreness vectors $\bm{c}$ close to the vector of all ones.
In fact, the limit case $q \rightarrow \infty$ restricts $\bm{c}$ to have unit maximum norm in which case \cref{eq:objective} is clearly maximised for $\bm{c}=\bm{1}\in\R^L$ due to the non-negativity of its summands.
Relatively small values of $q$, instead, lead to large ratios between individual layer weights.
Numerical experiments reported in \Cref{sec:numerical_experiments} indicate that the former choice is favourable for multiplex networks in which the core-periphery information is equally distributed across the layers.
The latter choice, in turn, is capable of automatically detecting and assigning low layer weights to noisy layers.

For the node coreness vector $\bm{x}$ the limiting behaviour $p\rightarrow\infty$ contradicts the goal of distinguishing between core and periphery nodes.
In fact, we observe over a wide range of numerical experiments that small values of $p$ such as $p=2$ lead to very good core-periphery partitions in single-layer and multiplex networks although this parameter setting does not guarantee unique globally optimal solutions.
In addition, multiplex networks show a remarkable robustness against uninformative noisy layers for small values of $p$, cf.~\Cref{sec:numerical_experiments_informative_vs_noise}.

\section{Measuring multiplex core-periphery structure and detecting optimal core sizes}\label{sec:measuring_coreness}

For single-layer networks, one possibility for judging the quality of a core-periphery partition is by visual inspection of the adjacency matrix with reordered rows and columns according to the permutation obtained by ranking the entries in the node coreness vector $\bm{x}$.
Good core-periphery partitions should be reflected in reordered adjacency matrices with sparsity patterns similar to the ideal block model, i.e., an L-shape.
The same visual measure can be applied to our multiplex framework in a layer-wise manner.
In addition, the layer coreness vector $\bm{c}$ indicates how much each layer contributes to the obtained reordering.
The spy plots in \Cref{fig:internet_p_22_q_2,fig:internet_p_22_q_22,fig:internet_p_2_q_2,fig:rp14_spy} show permuted adjacency matrices of the indicated layers of different multiplex networks alongside the layer coreness vector $\bm{c}$.

Besides the visual comparison of reordered adjacency matrices with the desired L-shape, we consider two quantitative measures to evaluate the quality of a given core-periphery partition.

\subsection*{Core-periphery profile}
The first relies on so-called core-periphery profiles, which are based on the persistence probability of random walks on single-layer networks \cite{rossa2013profiling}.
Considering subsets $S$ of the node index set $\{1, \dots , n\}$ containing the $m$ indices corresponding to the smallest entries in $\bm{x}$, the quantity
\begin{equation*}
P^{(k)}(S) = \frac{\sum_{ij\in S}\bm{A}_{ij}^{(k)}}{\sum_{i\in S}\sum_{j=1}^n \bm{A}_{ij}^{(k)}}
\end{equation*}
denotes the persistence probability of the set $S$ on layer $k$: the probability that a random walker in layer $k$ starting from an index in $S$, remains in $S$ after one step. 
In the ideal L-shape model, $P^{(k)}(S)$ is zero for all peripheral nodes and steeply increases as soon as an index belonging to a core node enters $S$.
Consequently, a late and sharp increase of $P^{(k)}(S)$ indicates a good partition.
\Cref{fig:RW_CP_profiles} shows the value $P^{(1)}(S)$ as a function of $|S|$, i.e., the number of the most peripheral nodes in the subset $S$ for the networks from \Cref{fig:internet_p_22_q_2,fig:internet_p_22_q_22,fig:internet_p_2_q_2}.

\subsection*{Multiplex QUBO objective function}
As a second quantitative measure of the quality of a given core-periphery partition, we consider function values of a modified objective function.
We do not utilise \cref{eq:objective} for this task since its values were found to be difficult to compare for vectors $\bm{x}$ and $\bm{c}$ obtained by different methods or by different parameters $p$ and $q$.
Instead, we adopt an objective function based on the  QUBO objective function from \cite{higham2022core}, evaluated on binary node coreness vectors.
For a binary vector $\bar{\bm x}\in \{0,1\}^n$, the single-layer QUBO objective function compares the present and the missing edges of the given network with those of an ideal L-shaped block model graph that has all edges connecting nodes $(i,j)$ such that $\bar{\bm x}_i+\bar{\bm x}_j >0$ and has no edges among nodes $(i,j)$ with $\bar{\bm x}_i+\bar{\bm x}_j =0$.
In the following, we propose a multiplex version of such a single-layer QUBO objective function.

Let $n_1^{(k)}=\sum_{i,j=1}^n \bm{A}_{ij}^{(k)}$ and $n_2^{(k)}=n^2-n_1^{(k)}$ denote the number of present and missing edges in layer $k$, respectively.
Since $\bm{A}^{(k)}$ is binary for all $k=1, \dots , L$, the matrix $\bm{1}\bm{1}^T - \bm{A}^{(k)}$ denotes the adjacency matrix of the complementary network of layer $k$.
Using the entries of the layer coreness vector $\bm{c}$ to weight present edges $\bm{A}_{ij}^{(k)}$ and missing edges $(1 - \bm{A}_{ij}^{(k)})$ and normalising both terms by $n_1^{(k)}$ and $n_2^{(k)}$, respectively, leads to the multiplex QUBO objective function
\begin{equation}\label{eq:qubo-multi-1}
\sum_{{i,j=1}}^n \left[ \left( \sum_{k=1}^L \bm{c}_k \frac{\bm{A}_{ij}^{(k)}}{n_1^{(k)}} \max \{ \bar{\bm{x}}_i, \bar{\bm{x}}_j\}  \right) + \left( \sum_{k=1}^L \bm{c}_k \frac{1 - \bm{A}_{ij}^{(k)}}{n_2^{(k)}}  (1-\max \{ \bar{\bm{x}}_i, \bar{\bm{x}}_j\})  \right) \right],
\end{equation}
cf.\ \cite{higham2022core} for more details on the single-layer case. Clearly, \cref{eq:qubo-multi-1} is large if, when $\bm c_k$ is large, i.e. we are looking at ``important'' layers, for every existing edge $\bm A_{ij}^{(k)}>0$  at least one of the two nodes $i,j$ is in the core, i.e., $\max \{ \bar{\bm{x}}_i, \bar{\bm{x}}_j\}=1$ and, at the same time, whenever the edge $i,j$ is not there, then both nodes are in the periphery, i.e., $\max \{ \bar{\bm{x}}_i, \bar{\bm{x}}_j\}=0$.

The vector $\bm{c}$ returned by \Cref{alg} is normalised in the prescribed $q$-norm.
In order to make multiplex QUBO values obtained for different choices of $q$ comparable, we introduce the normalisation $\bm{c}/\|\bm{c}\|_1$.
Replacing $\bm c$ with $\bm{c}/\|\bm{c}\|_1$ in \cref{eq:qubo-multi-1} and dropping all terms not depending on $\bar{\bm{x}}$, leads to
\begin{equation}\label{eq:QUBO_elementwise}
\sum_{k=1}^L \frac{\bm{c}_k}{\|\bm{c}\|_1} \sum_{i,j=1}^n \left( \bm{A}_{ij}^{(k)} \left( \frac{1}{n_1^{(k)}} + \frac{1}{n_2^{(k)}} \right) - \frac{1}{n_2^{(k)}} \right) \max \{ \bar{\bm{x}}_i, \bar{\bm{x}}_j\}.
\end{equation}

Using the relation $\max \{ \bar{\bm{x}}_i, \bar{\bm{x}}_j\} = \bar{\bm{x}}_i^2 + \bar{\bm{x}}_j^2 - \bar{\bm{x}}_i \bar{\bm{x}}_j$ for $\bar{\bm{x}}_i, \bar{\bm{x}}_j$ binary, we obtain
\begin{equation}\label{eq:QUBO_bilinear_form}
\sum_{k=1}^L \frac{\bm{c}_k}{\|\bm{c}\|_1} \left( \bar{\bm{x}}^T \bm{Q}^{(k)} \bar{\bm{x}} \right),
\end{equation}
where
\begin{equation}\label{eq:Qk}
\bm{Q}^{(k)} = 2 \left( \frac{1}{n_1^{(k)}} + \frac{1}{n_2^{(k)}} \right) \bm{D}^{(k)} - 2\frac{n-1}{n_2^{(k)}} \bm{I} - \left( \frac{1}{n_1^{(k)}} + \frac{1}{n_2^{(k)}} \right) \bm{A}^{(k)} + \frac{1}{n_2^{(k)}} (\bm{1}\bm{1}^T - \bm{I}),
\end{equation}
with $\bm{D}^{(k)}=\text{diag}(\bm{A}^{(k)}\bm{1})$, $\bm{I}\in\R^{n \times n}$ the identity matrix, and $\bm{1}\in\R^n$ the vector of all ones.

The obtained core-periphery quality function \eqref{eq:QUBO_bilinear_form} can now be used to quantify the quality of a multiplex core-periphery score pair $(\bm x,\bm c)\in \mathbb R^n_{>0} \times \mathbb R^L_{>0}$ via a sweeping procedure, which works as follows: we sort the entries of $\bm x$ in descending order, then form the sequence of binary vectors $\bar{\bm{x}}^{(s)}$, $s=1,\dots,n$ by assigning value $1$ to the largest $s$ entries of $\bm x$ and by setting the remaining entries to zero. Thus, we evaluate \eqref{eq:QUBO_bilinear_form} on each of these vectors and select the one achieving the largest value, which we use to quantify the quality of the score pair $(\bm x,\bm c)$. Furthermore, the corresponding optimal index $s^*$ is used to identify a partition into core and periphery sets, associated to the core-periphery score pair $(\bm x,\bm c)$.

The sequence of bilinear forms in \cref{eq:QUBO_bilinear_form} can be evaluated efficiently by an updating procedure since the inclusion of one additional non-zero entry in $\bar{\bm{x}}$ merely adds the sum of one row and column of $\bm{Q}^{(k)}$ on the non-zero indices of $\bar{\bm{x}}$ to one summand of \cref{eq:QUBO_bilinear_form}.
This way, the updating procedure only requires to access each entry of $\bm{Q}^{(k)}$ once.

We now provide the following interpretation of the value of \cref{eq:QUBO_bilinear_form} with the proof moved to \Cref{sec:appendix_proof_QUBO_bounds}.

\begin{theorem}\label{thm:QUBO_bounds}
	For $\bm{c}\in\R^L_{>0}$ and $\bar{\bm{x}}\in\{0,1\}^n$, we have the following bounds on the multiplex QUBO objective function
	\begin{equation}\label{eq:QUBO_bounds}
	-1 \leq \sum_{k=1}^L \frac{\bm{c}_k}{\|\bm{c}\|_1} \left( \bar{\bm{x}}^T \bm{Q}^{(k)} \bar{\bm{x}} \right) \leq 1.
	\end{equation}
	The optimal value $1$ is obtained only if the multiplex consists of perfectly  L-shaped layers, all having the same core size.
	Moreover, the maximum of \cref{eq:QUBO_bilinear_form} over the $n$ binary vectors $\bar{\bm{x}}$ is greater or equal to $0$.
\end{theorem}

Hence, the closer \cref{eq:QUBO_bilinear_form} is to $1$ the smaller the (reordered) network layers' deviation from the ideal L-shape is, indicating the presence of a more clear-cut multiplex core-periphery structure. 
\Cref{tab:QUBO_informative_vs_noise_1,tab:QUBO_informative_vs_noise_2,tab:QUBO_real_world_multiplexes} report maximal multiplex QUBO values for various multiplex networks alongside the optimal core size $s^\ast$.
Furthermore, the red lines in the spy plots in \Cref{fig:internet_p_22_q_2,fig:internet_p_22_q_22,fig:internet_p_2_q_2,fig:rp14_spy} indicate the index $s^\ast$ maximising the multiplex QUBO objective function.

\section{Numerical experiments}\label{sec:numerical_experiments}

We test \Cref{alg} on a range of artificial and real-world multiplex networks.
Numerical experiments on multiplex stochastic block model (SBM) networks \cite{holland1983stochastic} not reported in this paper show that prescribed core-periphery structures are reliably detected by the multiplex nonlinear spectral method (MP NSM).
This includes multiplex SBM networks with a varying degree of overlap of the sets of core nodes as well as additional noise layers.
Furthermore, a comparison with the multilayer degree (ML degree) approach from \cite{battiston2018multiplex} shows a high similarity of the performance of both methods on multiplex SBM networks.
Interpreting \Cref{alg} and the multilayer degree method as the corresponding single-layer method on aggregated networks, this behaviour is to be expected as it coincides with observations made in the single-layer case \cite{tudisco2019nonlinear}.

In our experiments, we use an AMD Ryzen 5 5600X 6-Core processor with $16$GB memory as well as \texttt{julia} version 1.4.1 \cite{bezanson2017julia}.
All \texttt{julia} codes for reproducing the results presented in this paper are publicly available under \url{https://github.com/COMPiLELab/MPNSM}.

\subsection{One informative, one noise layer}\label{sec:numerical_experiments_informative_vs_noise}

\setlength\tabcolsep{4pt}

\begin{table}
	\small
	\begin{center}
		\begin{tabular}{llcccccc}
			\hline\hline
			\multirow{2}{*}{Network}&&\multicolumn{3}{|c|}{Internet}&\multicolumn{3}{c|}{email-EU All}\\
			&&\multicolumn{1}{|c}{$0\%$ noise} & $10\%$ noise & \multicolumn{1}{c|}{$25\%$ noise} & $0\%$ noise & $10\%$ noise & \multicolumn{1}{c|}{$25\%$ noise}\\\hline\hline
			\multicolumn{2}{l}{Layer weight vector $\bm{c}$}&\multirow{2}{*}{$\begin{bmatrix}1\\0\end{bmatrix}$} & \multirow{2}{*}{$\begin{bmatrix}0.9967\\0.0816\end{bmatrix}$} & \multirow{2}{*}{$\begin{bmatrix}0.9797\\0.2007\end{bmatrix}$} & \multirow{2}{*}{$\begin{bmatrix}1\\0\end{bmatrix}$}&\multirow{2}{*}{$\begin{bmatrix}0.9975\\0.0706\end{bmatrix}$}&\multirow{2}{*}{$\begin{bmatrix}0.9832\\0.1826\end{bmatrix}$}\\
			($p=22, q=2$)&&&&&&&\\\hline\hline
			\multirow{2}{*}{\textbf{MP NSM}}&score&\second{0.8228}&\second{0.7604}&\second{0.6835}&\second{0.9773}&\second{0.9127}&\second{0.8242}\\
			&size&(1\,153)&(1\,161)&(1\,247)&(1\,891)&(1\,892)&(1\,926)\\\hline
			ML degree &score&0.8092&0.7483&0.6716&0.9694&0.9053&0.8177\\
			\textbf{(w/ MP NSM)}&size&(1\,121)&(1\,122)&(1\,089)&(1\,990)&(2\,005)&(2\,004)\\\hline
			h-index&score&0.5958&0.5503&0.4955&0.8363&0.7812&0.7052\\
			\textbf{(w/ MP NSM)}&size&(3\,031)&(3\,031)&(3\,031)&(5\,857)&(5\,857)&(5\,857)\\\hline
			EigA&score&0.6283&0.5805&0.5215&0.8930&0.8340&0.7533\\
			\textbf{(w/ MP NSM)}&size&(1\,196)&(1\,182)&(947)&(5\,127)&(5\,123)&(5\,105)\\\hline
			EigQ&score&0.4075&0.3768&0.3387&0.9091&0.8359&0.7337\\
			\textbf{(w/ MP NSM)}&size&(270)&(277)&(298)&(2\,958)&(5\,244)&(8\,577)\\\hline\hline
			\multicolumn{2}{l}{Layer weight vector $\bm{c}$}&\multirow{2}{*}{$\begin{bmatrix}1\\0\end{bmatrix}$} & \multirow{2}{*}{$\begin{bmatrix}1\\1\end{bmatrix}$} & \multirow{2}{*}{$\begin{bmatrix}1\\1\end{bmatrix}$} & \multirow{2}{*}{$\begin{bmatrix}1\\0\end{bmatrix}$}&\multirow{2}{*}{$\begin{bmatrix}1\\1\end{bmatrix}$}&\multirow{2}{*}{$\begin{bmatrix}1\\1\end{bmatrix}$}\\
			\multicolumn{2}{l}{($p=22$, eq.\ weights)}&&&&&&\\\hline\hline
			\multirow{2}{*}{NSM} &score&\second{0.8228}&0.5163&0.4658&\second{0.9773}&0.6947&0.5983\\
			&size&(1\,153)&(4\,108)&(4\,006)&(1\,891)&(33\,273)&(30\,104)\\\hline
			\multirow{2}{*}{ML degree}&score&0.8092&0.4472&0.4281&0.9694&0.5705&0.5207\\
			&size&(1\,121)&(4\,732)&(2\,806)&(1\,990)&(28\,827)&(25\,163)\\\hline
			\multirow{2}{*}{h-index}&score&0.5958&0.2953&0.3008&0.8363&0.4189&0.4178\\
			&size&(3\,031)&(3\,032)&(3\,032)&(5\,857)&(5\,858)&(5\,857)\\\hline
			\multirow{2}{*}{EigA}&score&0.6283&0.3136&0.3138&0.8930&0.4462&0.4465\\
			&size&(1\,196)&(943)&(941)&(5\,127)&(5\,103)&(5\,120)\\\hline
			\multirow{2}{*}{EigQ}&score&0.4075&0.2397&0.2203&0.9091&0.6272&0.5095\\
			&size&(270)&(867)&(787)&(2\,958)&(24\,032)&(31\,559)\\\hline\hline
			\multicolumn{2}{l}{Layer weight vector $\bm{c}$}&\multirow{2}{*}{$\begin{bmatrix}1\\0\end{bmatrix}$} & \multirow{2}{*}{$\begin{bmatrix}0.9999\\0.0008\end{bmatrix}$} & \multirow{2}{*}{$\begin{bmatrix}0.9999\\0.0022\end{bmatrix}$} & \multirow{2}{*}{$\begin{bmatrix}1\\0\end{bmatrix}$}&\multirow{2}{*}{$\begin{bmatrix}0.9999\\0.0002\end{bmatrix}$}&\multirow{2}{*}{$\begin{bmatrix}0.9999\\0.0007\end{bmatrix}$}\\
			\multicolumn{2}{l}{($p=q=2$)}&&&&&&\\\hline\hline
			\multirow{2}{*}{\textbf{MP NSM}}&score&\first{0.8243}&\first{0.8236}&\first{0.8225}&\first{0.9801}&\first{0.9799}&\first{0.9794}\\
			&size&(1\,306)&(1\,306)&(1\,305)&(1\,827)&(1\,827)&(1\,854)\\\hline
			ML degree&score&0.8092&0.8088&0.8068&0.9694&0.9692&0.9688\\
			\textbf{(w/ MP NSM)}&size&(1\,121)&(1\,122)&(1\,146)&(1\,990)&(2\,098)&(2\,004)\\\hline
			h-index&score&0.5958&0.5953&0.5945&0.8363&0.8362&0.8358\\
			\textbf{(w/ MP NSM)}&size&(3\,031)&(3\,031)&(3\,031)&(5\,857)&(5\,857)&(5\,857)\\\hline
			EigA&score&0.6283&0.6278&0.6269&0.8930&0.8928&0.8924\\
			\textbf{(w/ MP NSM)}&size&(1\,196)&(1\,196)&(1\,196)&(5\,127)&(5\,127)&(5\,127)\\\hline
			EigQ&score&0.4075&0.4072&0.4066&0.9091&0.8903&0.8632\\
			\textbf{(w/ MP NSM)}&size&(270)&(270)&(270)&(2\,958)&(5\,161)&(8\,519)\\\hline\hline
			\multicolumn{2}{l}{Layer weight vector $\bm{c}$}&\multirow{2}{*}{$\begin{bmatrix}1\\0\end{bmatrix}$} & \multirow{2}{*}{$\begin{bmatrix}1\\1\end{bmatrix}$} & \multirow{2}{*}{$\begin{bmatrix}1\\1\end{bmatrix}$} & \multirow{2}{*}{$\begin{bmatrix}1\\0\end{bmatrix}$} & \multirow{2}{*}{$\begin{bmatrix}1\\1\end{bmatrix}$} & \multirow{2}{*}{$\begin{bmatrix}1\\1\end{bmatrix}$}\\
			\multicolumn{2}{l}{($p=2$, eq.\ weights)}&&&&&&\\\hline\hline
			\multirow{2}{*}{NSM}&score&\first{0.8243}&0.5564&0.4872&\first{0.9801}&0.7237&0.6190\\
			&size&(1\,306)&(5\,727)&(5\,199)&(1\,827)&(27\,543)&(25\,824)\\\hline\hline
		\end{tabular}
	\end{center}
	\caption{Multiplex QUBO score and optimal core size $s^\ast$ for two two-layer multiplex networks consisting of one real-world single-layer network with strong core-periphery structure and one additional noise layer.
	The noise levels indicate the relative number of non-zero entries in the noise layer in comparison to the informative layer.
	Methods labeled with ``(w/ MP NSM)'' use traditional techniques from the literature combined with the optimised layer weights computed by our method MP NSM.}\label{tab:QUBO_informative_vs_noise_1}
\end{table}

Our first set of experiments considers artificially created multiplex networks that consist of one informative real-world single-layer network with strong core-periphery structure and a second layer containing noise, i.e., random edges with uniform probability.
The varying noise levels are chosen relative to the number of non-zero entries in the informative layer.
We consider the following four real-world single-layer networks.
For each network, we take the symmetrised and binarised version of the largest connected component of the informative layer and a symmetric binary noise layer of the same size.

\textbf{Internet} (as-22july06)\footnote{\label{suitesparse}Available in the SuiteSparse Matrix Collection \url{https://sparse.tamu.edu/}} \cite{newmandata}, internet snapshot with $96\,872$ edges representing hyperlinks between $n=22\,963$ websites.

\textbf{email-EUAll}\footnoteref{suitesparse} \cite{leskovec2007graph}, email network with $680\,720$ edges representing emails exchanged between $n=224\,832$ email addresses of a European research institution.

\textbf{Cardiff Tweets}\footnote{Available under \url{https://github.com/ftudisco/nonlinear-core-periphery}} \cite{grindrod2016comparison}, Twitter network with $8\,888$ edges representing mentions between $n=2\,685$ users.

\textbf{Yeast PPI}\footnoteref{suitesparse} \cite{bu2003topological}, protein-protein interaction network with $13\,828$ edges representing interactions between $n=2\,361$ proteins.

\setlength\tabcolsep{2pt}

\begin{figure}
	\begin{center}
		\begin{tabular}{lcccc}
			\hline\hline
			&$0\%$ noise & $10\%$ noise & $25\%$ noise & $50\%$ noise\\\hline\hline
			&$\bm{c}=[1, 0]^T$ & $\bm{c}=[0.9967, 0.0816]^T$ & $\bm{c}=[0.9797, 0.2007]^T$ & $\bm{c}=[0.9210, 0.3896]^T$\\\hline\hline
			&&&&\\[-10pt]
			\parbox[t]{2mm}{\multirow{3}{*}{\rotatebox[origin=c]{90}{\hspace{-40pt} MP NSM}}} & \multirow{2}{85pt}{\centering\includegraphics[width=.235\textwidth,clip,trim=50pt 0pt 50pt 10pt]{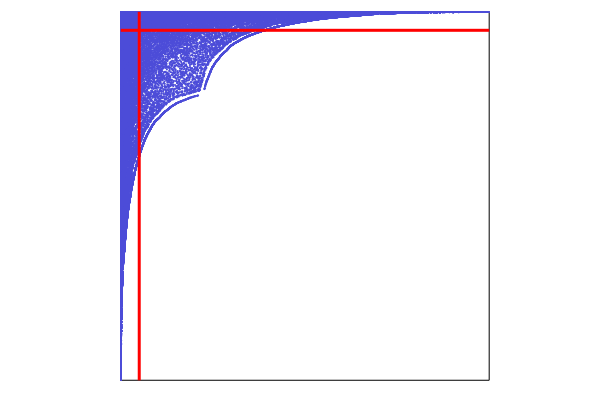}} & \multirow{2}{85pt}{\centering\includegraphics[width=.235\textwidth,clip,trim=50pt 0pt 50pt 10pt]{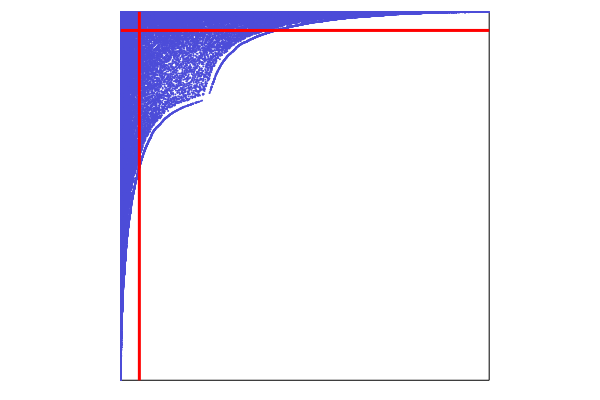}} & \multirow{2}{85pt}{\centering\includegraphics[width=0.235\textwidth,clip,trim=50pt 0pt 50pt 10pt]{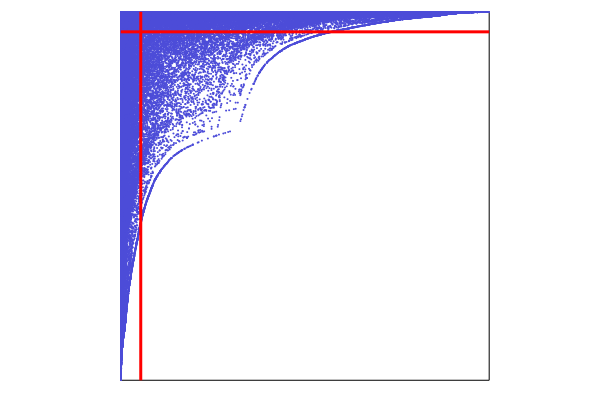}} & \multirow{2}{85pt}{\centering\includegraphics[width=0.235\textwidth,clip,trim=50pt 0pt 50pt 10pt]{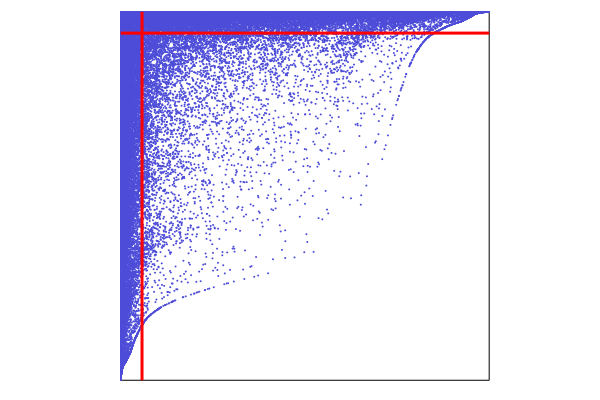}}\\
			&&&&\\
			&&&&\\
			&&&&\\
			&&&&\\[-2pt]
			&&&&\\\hline
			&&&&\\[-10pt]
			\parbox[t]{2mm}{\multirow{3}{*}{\rotatebox[origin=c]{90}{\hspace{-40pt} ML degree opt.}}} & \multirow{2}{85pt}{\centering\includegraphics[width=.235\textwidth,clip,trim=50pt 0pt 50pt 10pt]{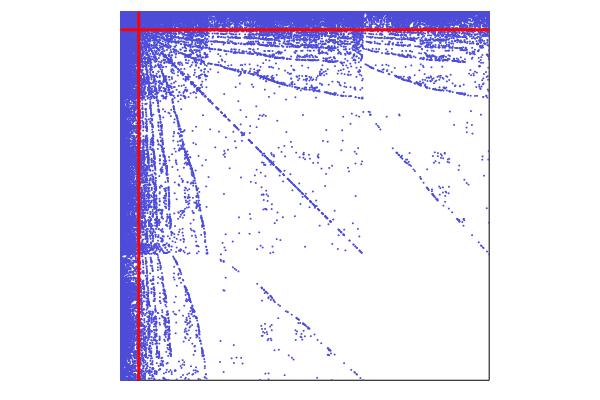}} & \multirow{2}{85pt}{\centering\includegraphics[width=.235\textwidth,clip,trim=50pt 0pt 50pt 10pt]{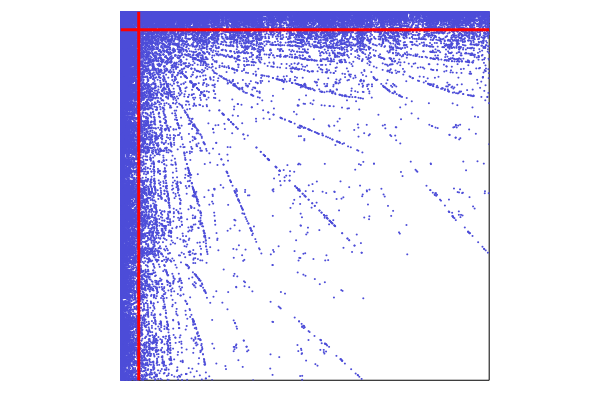}} & \multirow{2}{85pt}{\centering\includegraphics[width=0.235\textwidth,clip,trim=50pt 0pt 50pt 10pt]{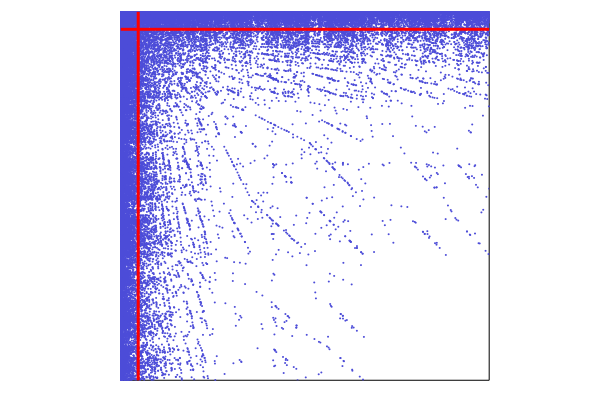}} & \multirow{2}{85pt}{\centering\includegraphics[width=0.235\textwidth,clip,trim=50pt 0pt 50pt 10pt]{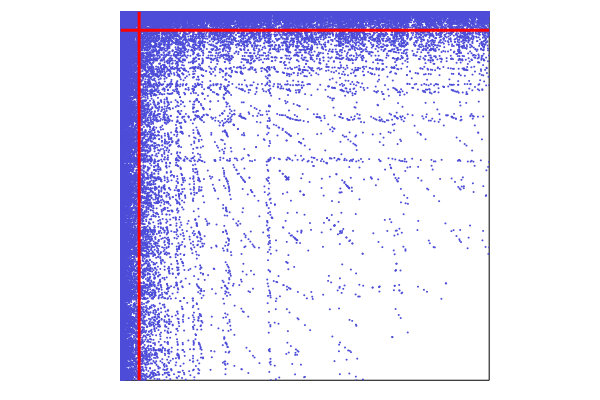}}\\
			&&&&\\
			&&&&\\
			&&&&\\
			&&&&\\[-2pt]
			&&&&\\\hline\hline
			&$\bm{c}=[1,0]^T$ & $\bm{c} = \bm{1}$&$\bm{c} = \bm{1}$&$\bm{c} = \bm{1}$\\\hline\hline
			&&&&\\[-10pt]
			\parbox[t]{2mm}{\multirow{3}{*}{\rotatebox[origin=c]{90}{\hspace{-40pt} NSM aggr.\ eq.}}} & \multirow{2}{85pt}{\centering\includegraphics[width=.235\textwidth,clip,trim=50pt 0pt 50pt 10pt]{graphics/NSM_internet_single-layer_p_22.png}} & \multirow{2}{85pt}{\centering\includegraphics[width=.235\textwidth,clip,trim=50pt 0pt 50pt 10pt]{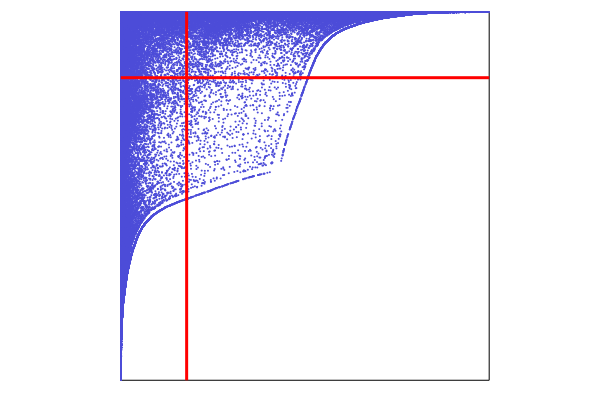}} & \multirow{2}{85pt}{\centering\includegraphics[width=0.235\textwidth,clip,trim=50pt 0pt 50pt 10pt]{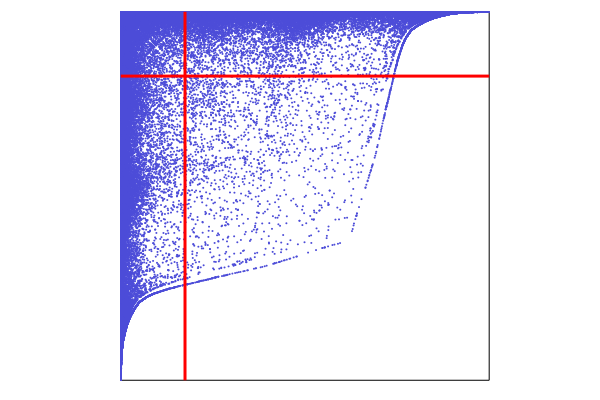}}	& \multirow{2}{85pt}{\centering\includegraphics[width=0.235\textwidth,clip,trim=50pt 0pt 50pt 10pt]{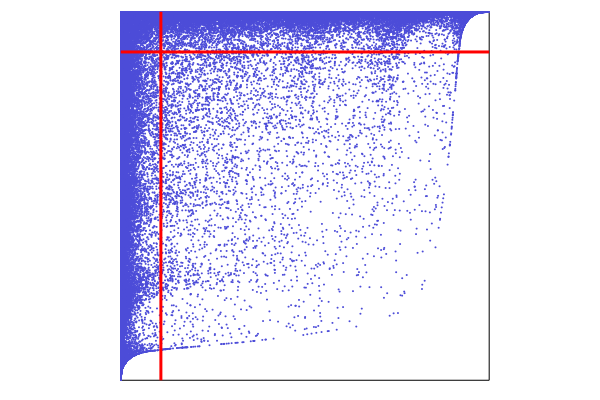}}\\
			&&&&\\
			&&&&\\
			&&&&\\
			&&&&\\[-2pt]
			&&&&\\\hline
			&&&&\\[-10pt]
			\vspace{10pt}\parbox[t]{2mm}{\multirow{3}{*}{\rotatebox[origin=c]{90}{\hspace{-40pt} ML degree eq.}}} & \multirow{2}{85pt}{\centering\includegraphics[width=.235\textwidth,clip,trim=50pt 0pt 50pt 10pt]{graphics/MLdegree_internet_single-layer.png}} & \multirow{2}{85pt}{\centering\includegraphics[width=.235\textwidth,clip,trim=50pt 0pt 50pt 10pt]{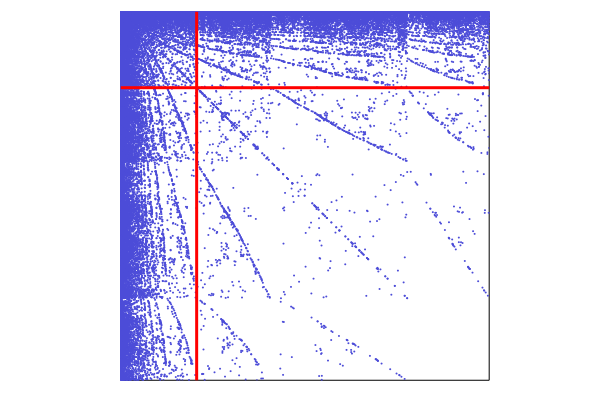}} & \multirow{2}{85pt}{\centering\includegraphics[width=0.235\textwidth,clip,trim=50pt 0pt 50pt 10pt]{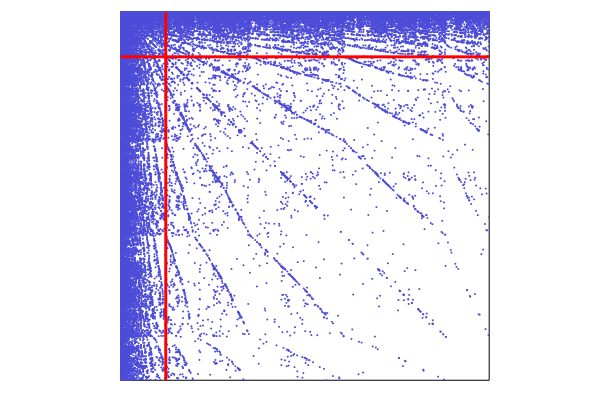}}	& \multirow{2}{85pt}{\centering\includegraphics[width=0.235\textwidth,clip,trim=50pt 0pt 50pt 10pt]{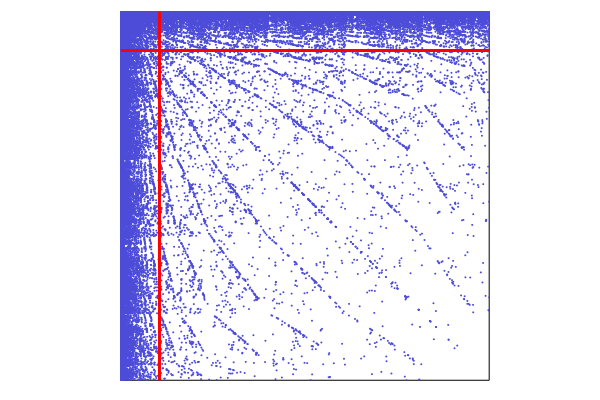}}\\[-12pt]
			&&&&\\
			&&&&\\
			&&&&\\
			&&&&\\
			&&&&\\\hline\hline
		\end{tabular}
	\end{center}
	\caption{Reordered adjacency matrices of the informative layer of the two-layer Internet multiplex network for the parameters $\alpha=10, p=22$, and $q=2$ for various levels of noise in the second uninformative layer as well as optimised and equal layer weights.}\label{fig:internet_p_22_q_2}
\end{figure}

\begin{figure}
	\begin{center}
		\begin{tabular}{ccccc}
			\hline\hline
			& $0\%$ noise & $10\%$ noise & $25\%$ noise & $50\%$ noise\\\hline\hline
			& $\bm{c}=[1,0]^T$ & $\bm{c} = [0.9999, 0.0008]^T$ & $\bm{c} = [0.9999, 0.0022]^T$ & $\bm{c} = [0.9999, 0.0038]^T$\\\hline\hline
			&&&&\\[-10pt]
			\parbox[t]{2mm}{\multirow{3}{*}{\rotatebox[origin=c]{90}{\hspace{-40pt} MP NSM}}} & \multirow{2}{85pt}{\centering\includegraphics[width=.235\textwidth,clip,trim=50pt 0pt 50pt 10pt]{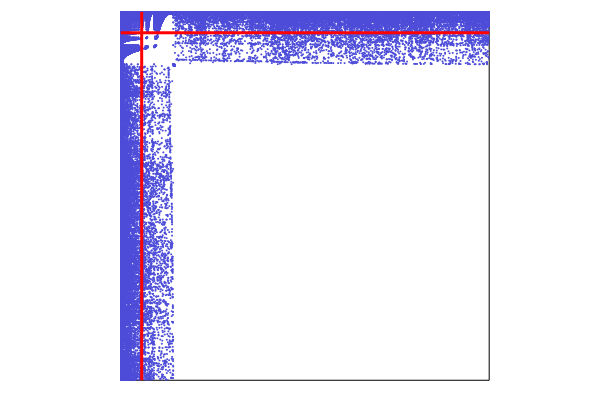}} & \multirow{2}{85pt}{\centering\includegraphics[width=.235\textwidth,clip,trim=50pt 0pt 50pt 10pt]{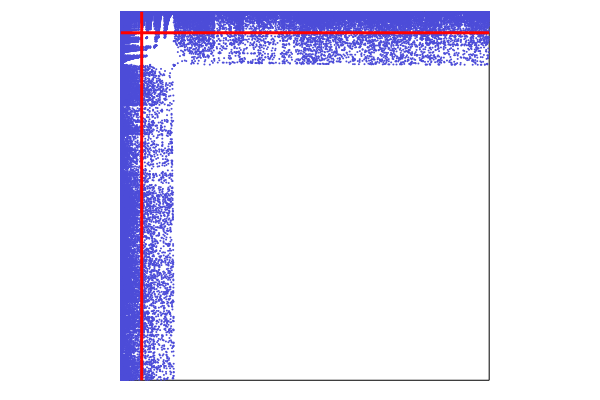}}	& \multirow{2}{85pt}{\centering\includegraphics[width=0.235\textwidth,clip,trim=50pt 0pt 50pt 10pt]{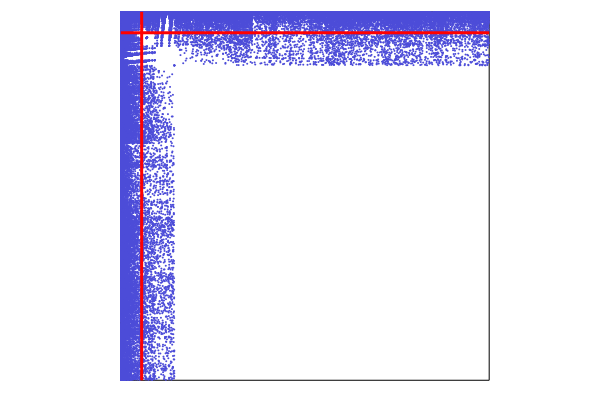}} & \multirow{2}{85pt}{\centering\includegraphics[width=0.235\textwidth,clip,trim=50pt 0pt 50pt 10pt]{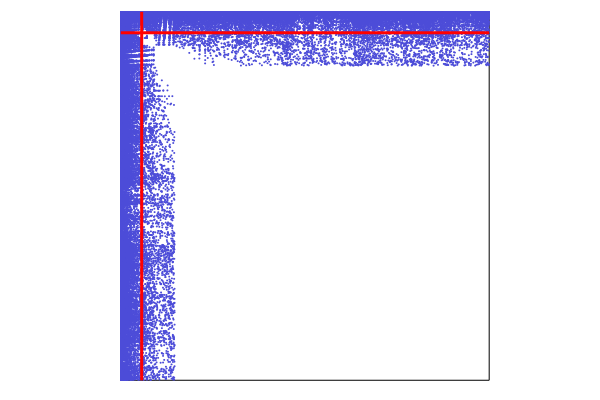}}\\
			&&&&\\
			&&&&\\
			&&&&\\
			&&&&\\[-2pt]
			&&&&\\\hline
			&&&&\\[-10pt]
			\parbox[t]{2mm}{\multirow{3}{*}{\rotatebox[origin=c]{90}{\hspace{-40pt} ML degree opt.}}} & \multirow{2}{85pt}{\centering\includegraphics[width=.235\textwidth,clip,trim=50pt 0pt 50pt 10pt]{graphics/MLdegree_internet_single-layer.png}} & \multirow{2}{85pt}{\centering\includegraphics[width=.235\textwidth,clip,trim=50pt 0pt 50pt 10pt]{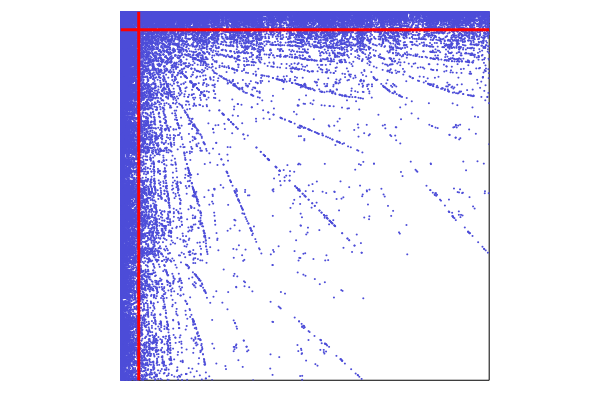}} & \multirow{2}{85pt}{\centering\includegraphics[width=0.235\textwidth,clip,trim=50pt 0pt 50pt 10pt]{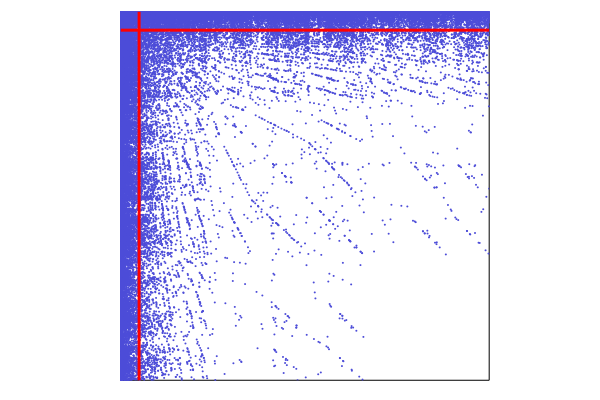}}	& \multirow{2}{85pt}{\centering\includegraphics[width=0.235\textwidth,clip,trim=50pt 0pt 50pt 10pt]{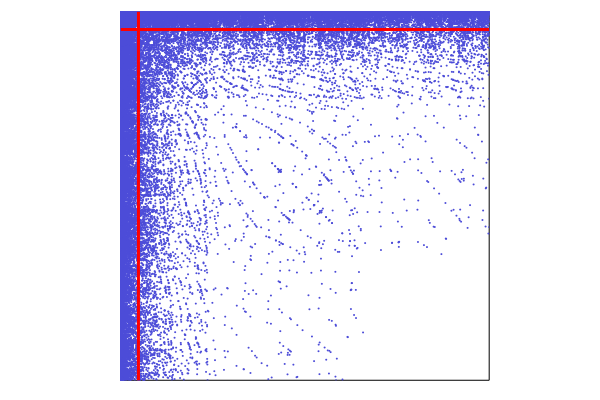}}\\
			&&&&\\
			&&&&\\
			&&&&\\
			&&&&\\[-2pt]
			&&&&\\\hline\hline
			& $\bm{c}=[1,0]^T$ &$\bm{c} = \bm{1}$&$\bm{c} = \bm{1}$&$\bm{c} = \bm{1}$\\\hline\hline
			&&&&\\[-10pt]
			\parbox[t]{2mm}{\multirow{3}{*}{\rotatebox[origin=c]{90}{\hspace{-40pt} NSM aggr.\ eq.}}} & \multirow{2}{85pt}{\centering\includegraphics[width=.235\textwidth,clip,trim=50pt 0pt 50pt 10pt]{graphics/NSM_internet_single-layer_p_2.png}} & \multirow{2}{85pt}{\centering\includegraphics[width=.235\textwidth,clip,trim=50pt 0pt 50pt 10pt]{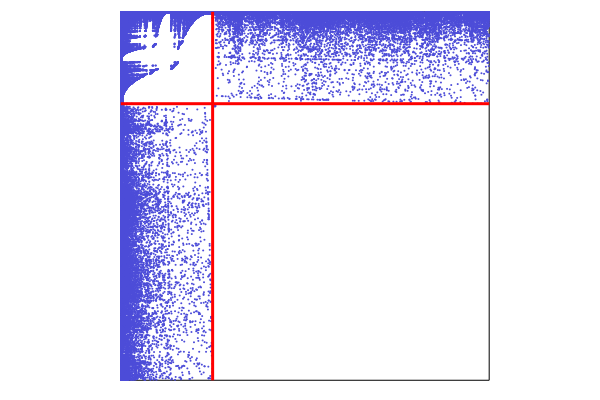}} & \multirow{2}{85pt}{\centering\includegraphics[width=0.235\textwidth,clip,trim=50pt 0pt 50pt 10pt]{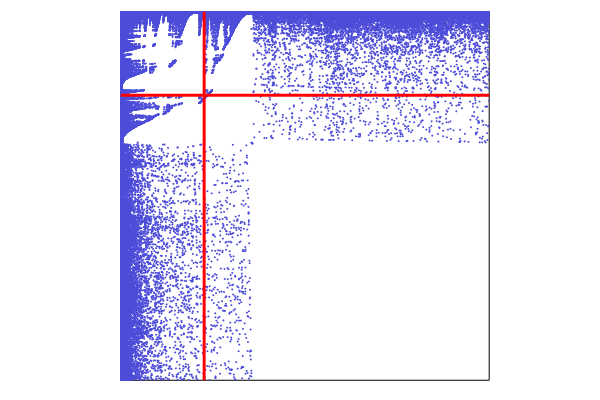}} & \multirow{2}{85pt}{\centering\includegraphics[width=0.235\textwidth,clip,trim=50pt 0pt 50pt 10pt]{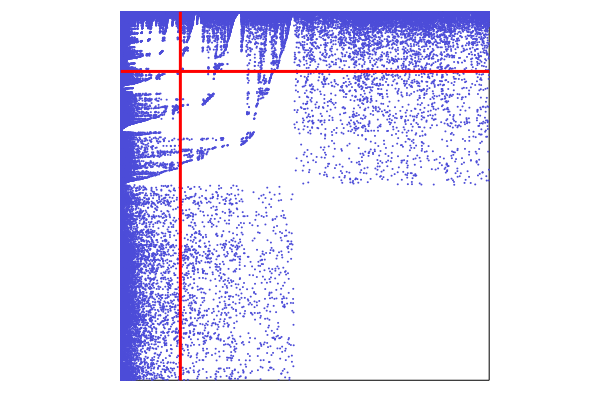}}\\
			&&&&\\
			&&&&\\
			&&&&\\
			&&&&\\[-2pt]
			&&&&\\\hline\hline
		\end{tabular}
	\end{center}
	\caption{Reordered adjacency matrices of the informative layer of the two-layer Internet multiplex network for the parameters $\alpha=10$ and $p=q=2$ for various levels of noise in the second uninformative layer as well as optimised and equal layer weights.
		The results of ML degree eq.\ coincide with those displayed in \Cref{fig:internet_p_22_q_2}.}\label{fig:internet_p_2_q_2}
\end{figure}

In all our experiments, we compare the performance of the optimised layer weight vector $\bm{c}$ returned by \Cref{alg} for the specified choice of $q$ with equal layer weights corresponding to the limit case $q\rightarrow\infty$.

In \Cref{tab:QUBO_informative_vs_noise_1} as well as \Cref{tab:QUBO_informative_vs_noise_2} in the appendix, we compare the maximal multiplex QUBO values and the corresponding optimal core sizes $s^\ast$ for the four artificial multiplex networks and for five different core-periphery detection methods, cf.\ \Cref{sec:measuring_coreness}.
In addition to our method (MP NSM) and the multilayer degree method (ML degree) \cite{battiston2018multiplex}, we report results of the core-periphery detection methods h-index \cite{lu2016h}, EigA \cite{borgatti2000models}, EigQ \cite{higham2022core}, and single-layer NSM \cite{tudisco2019nonlinear} applied to the aggregated networks using the indicated layer weight vector $\bm{c}$.
We observe throughout all experiments that MP NSM with $p=q=2$ obtains the best and MP NSM with $p=22$ and $q=2$ the second best multiplex QUBO scores indicating reordered adjacency matrices close to the ideal L-shape.
Maximal multiplex QUBO scores of ML degree typically range slightly below those of MP NSM obtained with the same layer weight vector $\bm{c}$.
Note that the performance of ML degree was significantly improved based on the layer weights optimised by \Cref{alg}.
Without these optimal weights, ML degree performs distinctively worse than MP NSM, cf.\ e.g., the multiplex QUBO scores in \Cref{tab:QUBO_informative_vs_noise_1,tab:QUBO_informative_vs_noise_2} for equal weights.

\begin{figure}
	\begin{center}
		\begin{tabular}{>{\centering}m{0.3\textwidth}>{\centering}m{0.3\textwidth}>{\centering\arraybackslash}m{0.3\textwidth}}
			\hline\hline
			$10\%$ noise & $25\%$ noise & $50\%$ noise\\\hline\hline
		\end{tabular}
		\subfloat[$p=22, q=2$]{
			\centering
			\shortstack{
				\includegraphics[width=.32\textwidth]{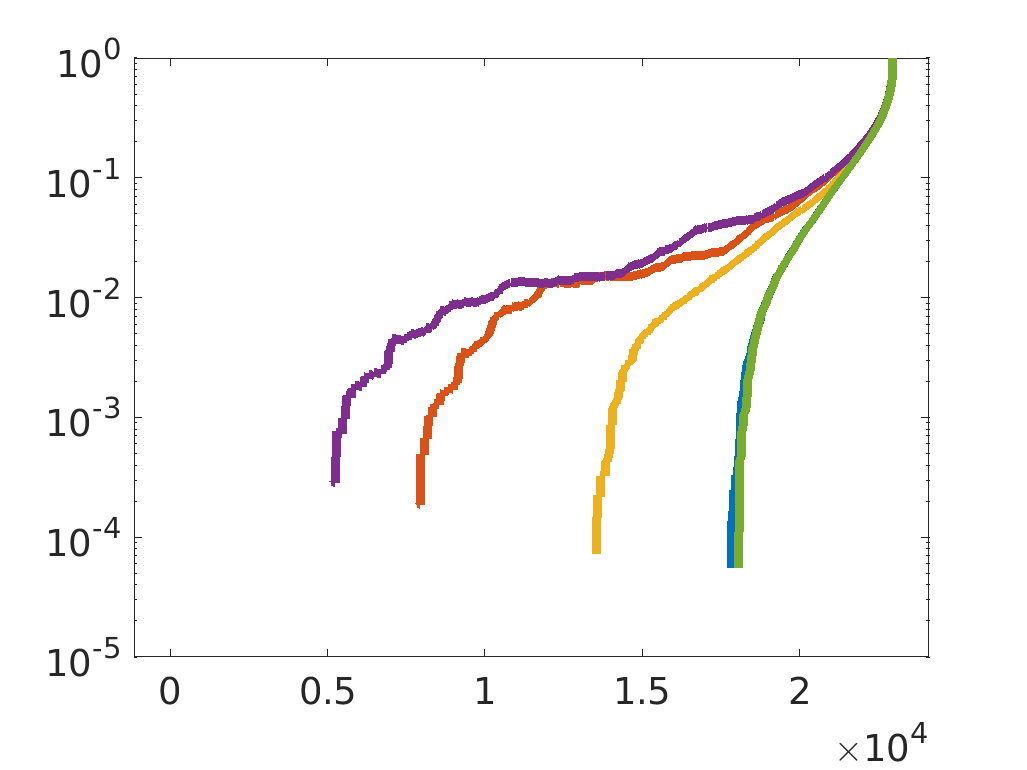} \includegraphics[width=.32\textwidth]{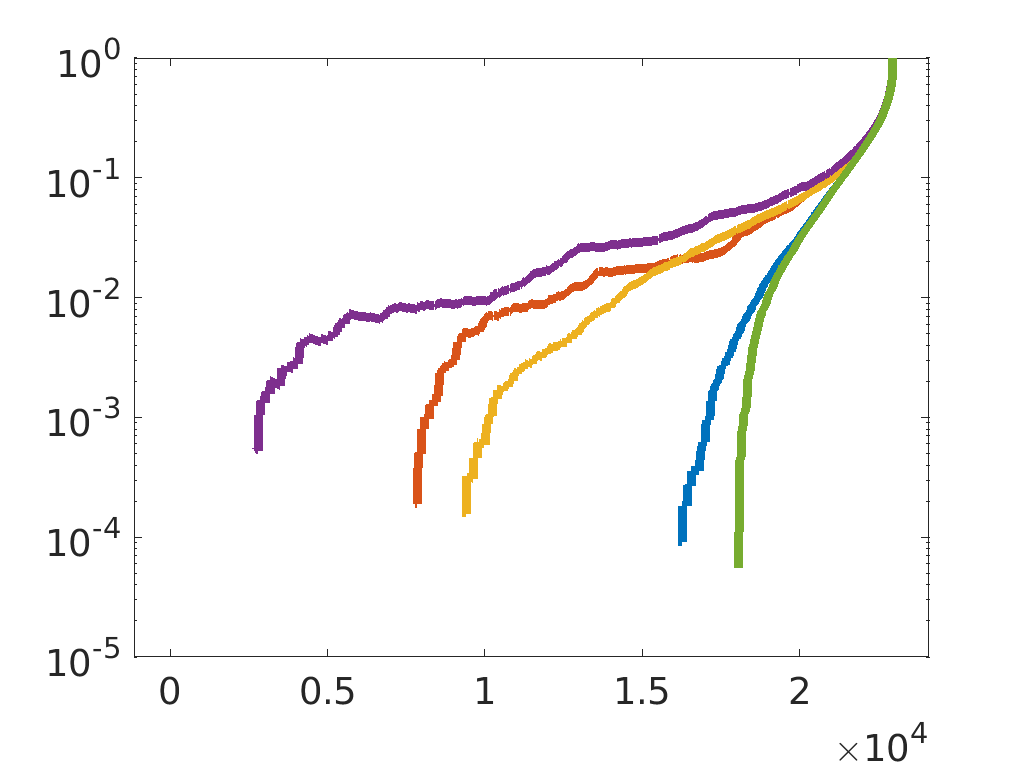} \includegraphics[width=.32\textwidth]{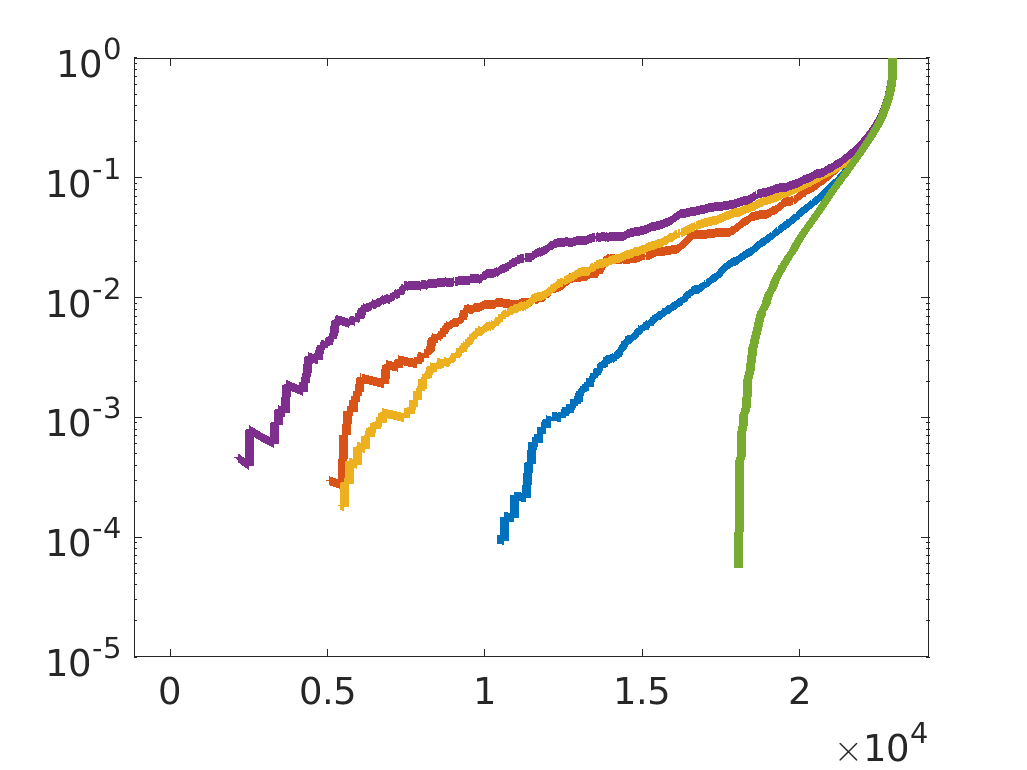}}
		}
		
		\subfloat[$p=q=22$]{
			\centering
			\shortstack{
				\includegraphics[width=.32\textwidth]{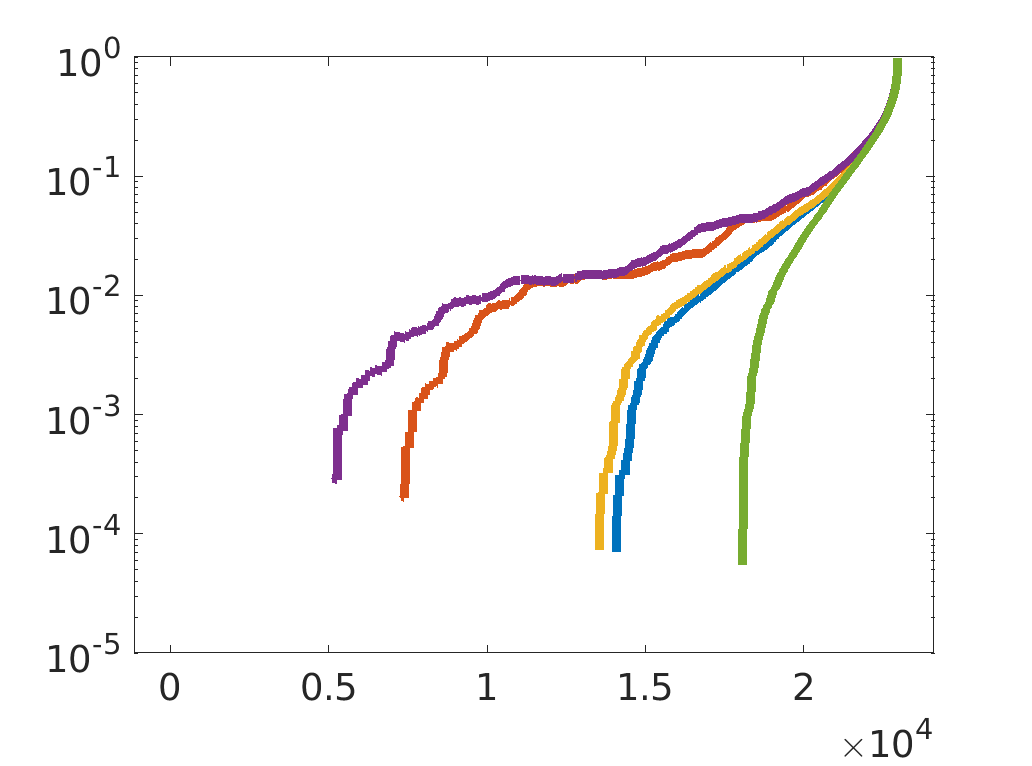} \includegraphics[width=.32\textwidth]{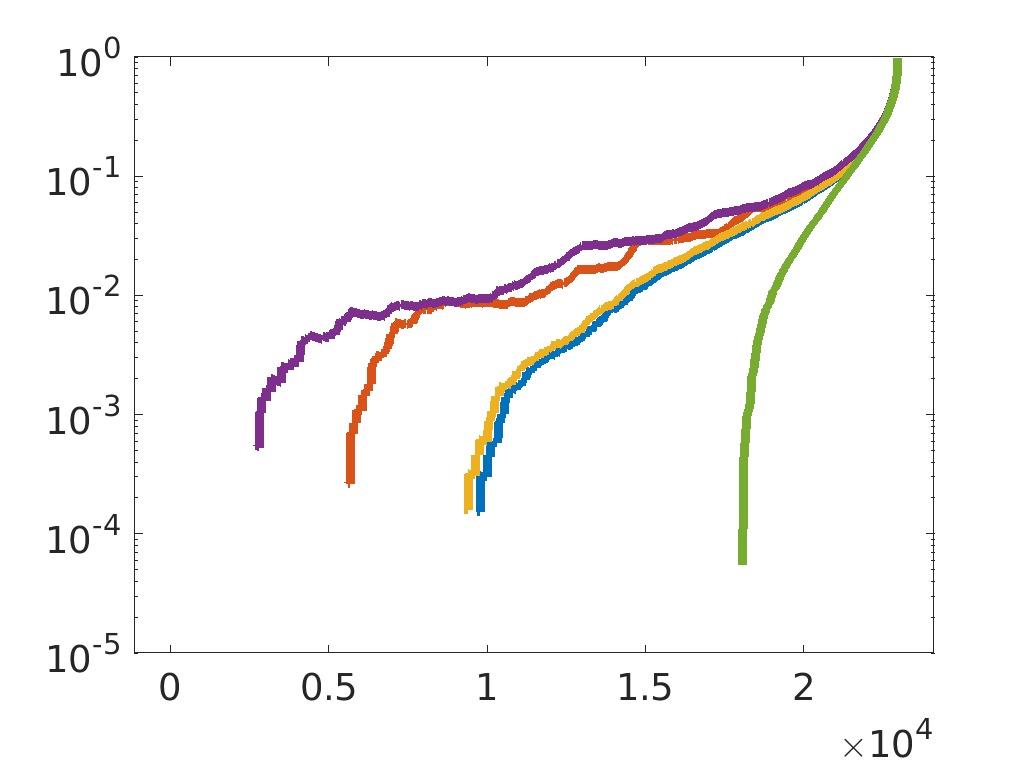} \includegraphics[width=.32\textwidth]{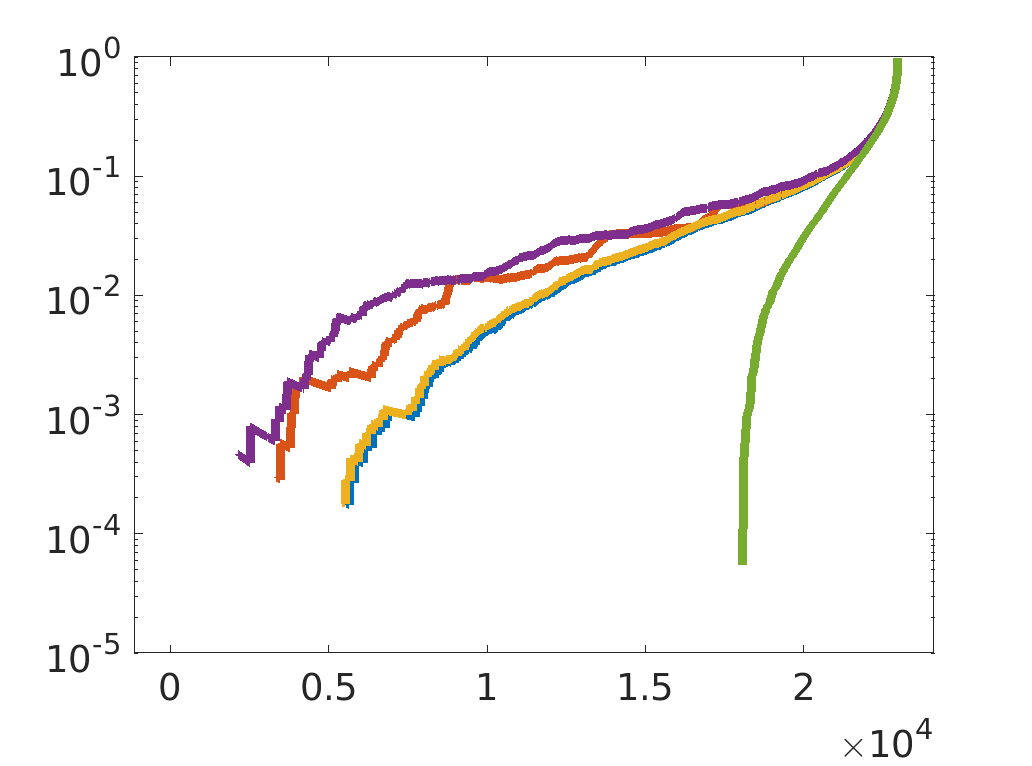}}
		}
		
		\subfloat[$p=q=2$]{
			\centering
			\shortstack{
				\includegraphics[width=.32\textwidth]{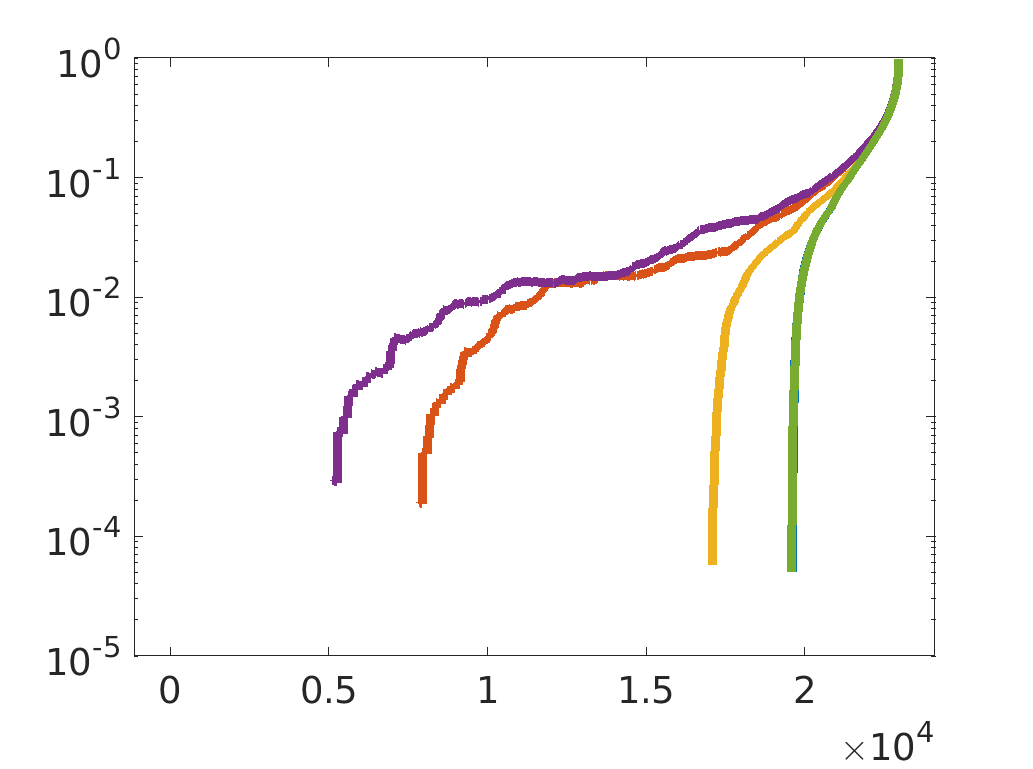} \includegraphics[width=.32\textwidth]{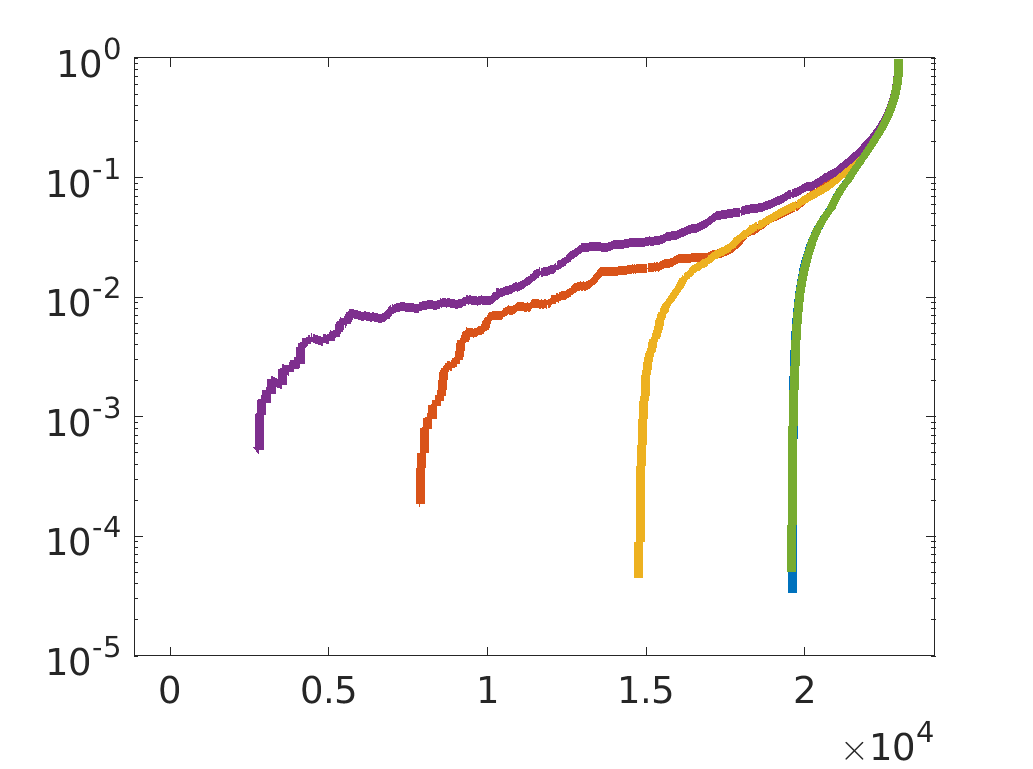} \includegraphics[width=.32\textwidth]{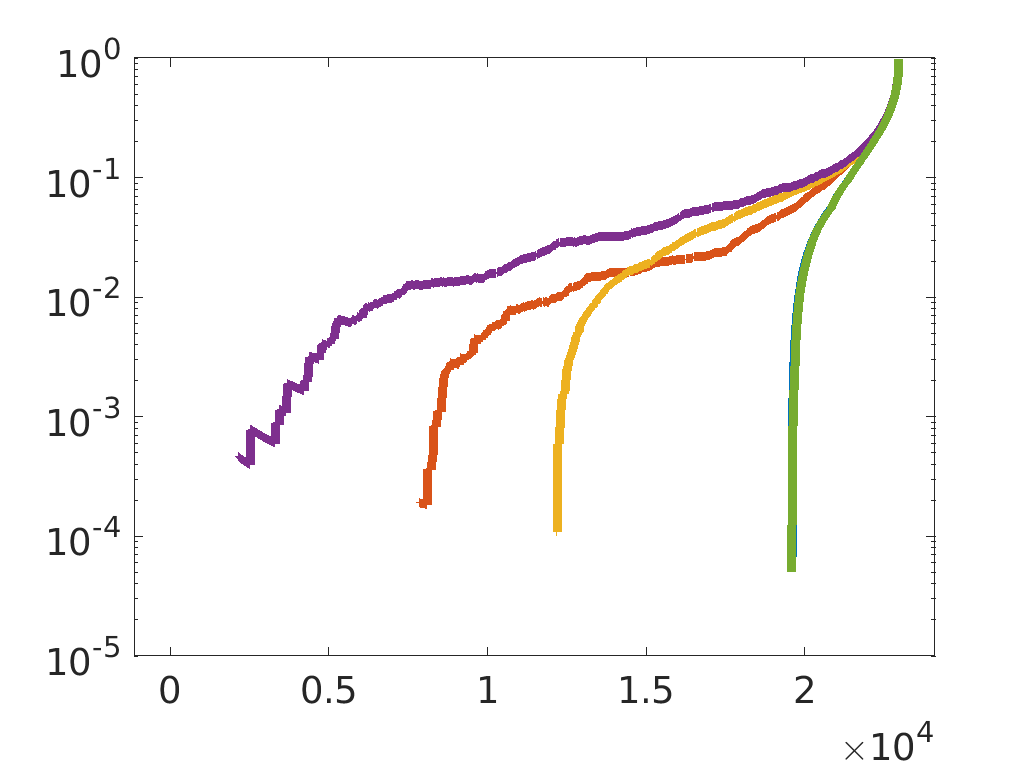}}
		}~\\
		\vspace{10pt}\includegraphics[width=.95\textwidth]{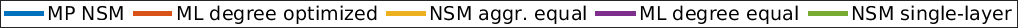}
	\end{center}
	\caption{Single-layer core-periphery profiles on the informative layer of the two-layer Internet multiplex network for different choices of the parameters $p$ and $q$.
		In panel (c), the green and blue curves of MP NSM and NSM single-layer lie almost on top of each other.}\label{fig:RW_CP_profiles}
\end{figure}

As is to be expected, an increasing noise level negatively affects the correct recovery of the core-periphery structure of the informative layer.
Furthermore, we observe significantly better results with the layer weight vectors $\bm{c}$ optimised by MP NSM compared to equal weights throughout all methods since the optimised weights assign less weight to the uninformative noise layer.
The parameter setting $p=q=2$ that only guarantees convergence to local optima proves to be particularly robust against noise.
One consequence of these observations is that our novel optimised layer weights significantly improve the multilayer degree method \cite{battiston2018multiplex}.

As described in \Cref{sec:measuring_coreness}, alternative means of assessing the quality of a given core-periphery partition are by visual inspection of the reordered adjacency matrices as well as by random walk core-periphery profiles.
Since by construction, we have one informative and one noisy layer in the experiments of this subsection, we apply both techniques only to the informative layer to assess the negative impact of the additional noise layer.

\Cref{fig:internet_p_22_q_2,fig:internet_p_2_q_2} as well as \Cref{fig:internet_p_22_q_22} in the appendix show sparsity plots of the reordered Internet network with one additional noise layer for different choices of the parameters $p$ and $q$ in the multiplex nonlinear spectral method.
ML degree opt.\ refers to \cref{eq:coreness_linear_combination} with $\bm{c}$ optimised by MP NSM while NSM aggr.\ eq.\ and ML degree eq.\ refer to the single-layer versions of the respective method applied to the aggregated network using equal layer weights.
The red lines indicate the ideal core size $s^\ast$.
The figures show that throughout all considered parameter settings, the reorderings of MP NSM are visually closer to the ideal L-shape than those of ML degree.
Additionally, the multiplex QUBO results from \Cref{tab:QUBO_informative_vs_noise_1,tab:QUBO_informative_vs_noise_2} are confirmed in the sense that the optimised layer weights provide visually superior results in comparison to equal weights.

The parameter choice $p=22$ in \Cref{fig:internet_p_22_q_2,fig:internet_p_22_q_22} satisfies the assumptions of \Cref{thm:unique_solution} and hence guarantees display of the globally optimal solution.
The relatively large parameter $q=22$ leads to layer weight vectors $\bm{c}$ that are relatively close to equal weights, e.g., the MP NSM results in \Cref{fig:internet_p_22_q_22} are similar to the NSM aggr.\ eq.\ results in \Cref{fig:internet_p_22_q_2}.
We mention again that we obtain $\bm{c}=\bm{1}$ in the limit case $q \rightarrow \infty$.

\Cref{fig:internet_p_2_q_2} uses the parameters $p=q=2$, which do not guarantee convergence to a global optimum.
Instead, we display the locally optimal solution as ensured by \Cref{thm:local_optimum} obtained with the starting vectors $\bm{x}_0=\bm{1}\in\R^n$ and $\bm{c}_0=\bm{1}\in\R^L$.
In practice, different positive starting vectors $\bm{x}_0$ and $\bm{c}_0$ lead to slightly different solution vectors $\bm{x}$ and $\bm{c}$, but we found them to give rise to very similar reordered adjacency matrices.
The case $p=q=2$ yields qualitatively different results in the sense that more L-shaped reorderings are obtained.
Remarkably, this parameter setting is extremely robust against noise as the optimised layer weight vector $\bm{c}$ remains close to the first unit vector for all considered noise levels.
In terms of the optimal core size $s^\ast$, which the single-layer NSM and degree method agree on ranging between $1\,100$ and $1\,200$, we see that the additional noise layer obfuscates this optimal core size in all settings but in those where the optimised layer weights obtained by $q=2$ are used in either method.

Finally, the visual observations are quantitatively confirmed by the random walk core-periphery profiles defined in \Cref{sec:measuring_coreness} applied to the informative layer.
A good core-periphery partition is characterised by a sharp increase in the random walk profile for a large index.
\Cref{fig:RW_CP_profiles} shows that for this measure, the layer weights optimised by the multiplex nonlinear spectral method outperform equal weights for MP NSM and ML degree in all parameter settings.
The random walk core-periphery profiles of MP NSM with $p=22$ are degrading with increasing noise level while they are visually indistinguishable from the single-layer result for $p=2$.
In particular, the single-layer random walk core-periphery profile obtained with $p=2$ is superior to that of $p=22$ although unique global optimality is only guaranteed for $p=22$.

The runtime of \Cref{alg} for the computationally most expensive experiment in this section, i.e., the email-EU All multiplex network with $25\%$ noise is $2.25$ seconds for $28$ iterations with the parameters $p=22, q=2$, as well as $4.34$ seconds for $59$ iterations with the parameters $p=2, q=2$.

\subsection{Real-world multiplex networks}\label{sec:numerical_experiments_real_world}

The second set of experiments is performed on real-world multiplex networks.
This situation is challenging since it is generally not known a-priori whether all layers contain a core-periphery structure and, in case  they do, to what extent the sets of core nodes overlap.
We inspected random walk core-periphery profiles on individual layers of several multiplex networks for different parameters $q$ in the multiplex nonlinear spectral method leading to layer weight vectors $\bm{c}$ ranging between being close to a unit vector and being close to the vector of all ones.
We found the best choice of $q$ to be dataset-dependent, i.e., small values of $q$ perform well if noise is present in some layers and the corresponding layer weight is close to $0$ while large values of $q$ perform well if the informativity is equally distributed across the layers.
It was often observed that the improvement of the core-periphery structure in one layer entails its deterioration in other layers.

In this section, we consider the following seven real-world multiplex networks with symmetrised and binarised layers.

\textit{Arabidopsis multiplex GPI network} (\textbf{Arabid.)}\footnote{\label{manliodata}Available under \url{https://manliodedomenico.com/data.php}} \cite{stark2006biogrid,de2015structural}, genetic interaction network with $L=7$ different types of interactions between the $n=6\,980$ genes of ``Arabidopsis Thaliana''.

\textit{Arxiv netscience multiplex} (\textbf{Arxiv})\footnoteref{manliodata} \cite{de2015identifying}, scientific collaboration network with $L=13$ different arXiv categories and co-authorship relations between $n=14\,489$ network scientists.

\begin{sloppypar}
	\textit{Drosophila multiplex GPI network} (\textbf{Drosoph.})\footnoteref{manliodata} \cite{stark2006biogrid,de2015structural}, genetic interaction network with $L=7$ different types of interactions between the $n=8\,215$ genes of ``Drosophila Melanogaster''.
\end{sloppypar}

\textit{EU-Air transportation multiplex} (\textbf{EU air.})\footnoteref{manliodata} \cite{cardillo2013emergence}, European airline transportation network with $L=37$ different airlines and their flight connections between $n=417$ European airports.

\textit{Homo multiplex GPI network} (\textbf{Homo})\footnoteref{manliodata} \cite{stark2006biogrid,de2015muxviz}, genetic interaction network with $L=7$ different types of interactions between the $n=18\,222$ genes of ``Homo Sapiens''.

\textit{Twitter Rana Plaza year $2013$} (\textbf{TRP '13})\footnote{\label{twitterrp}Available under \url{https://github.com/KBergermann/Twitter-Rana-Plaza}} \cite{bergermann2023twitter}, Twitter network with $L=3$ different types of interactions between $n=9\,925$ users.

\textit{Twitter Rana Plaza year $2014$} (\textbf{TRP '14})\footnoteref{twitterrp} \cite{bergermann2023twitter}, Twitter network with $L=3$ different types of interactions between $n=14\,866$ users.

\setlength\tabcolsep{5pt}

\begin{table}
	\small
	\begin{center}
		\begin{tabular}{llccccccc}
			\hline\hline
			Network&&Arabid.\ &Arxiv&Drosoph.\ &EU Air.\ &Homo&TRP '13&TRP '14\\\hline\hline
			\multicolumn{2}{l}{$p=22$, opt.\ weights}&&&&&&&\\\hline\hline
			\multirow{2}{*}{\textbf{MP NSM}}&score&0.6431&0.4105&0.5210&0.5768&0.6974&\second{0.8154}&\second{0.7499}\\
			&size&(746)&(1\,517)&(1\,220)&(67)&(1\,522)&(524)&(1\,078)\\\hline
			ML degree&score&0.6207&0.3930&0.5105&0.5626&0.6878&0.7921&0.7260\\
			\textbf{(w/ MP NSM)}&size&(581)&(1\,408)&(1\,239)&(68)&(1\,709)&(408)&(1\,046)\\\hline
			h-index&score&0.1526&0.1792&0.0954&0.4421&0.4101&0.2334&0.2133\\
			\textbf{(w/ MP NSM)}&size&(939)&(2\,063)&(1\,524)&(83)&(4\,156)&(1\,538)&(2\,206)\\\hline
			EigA&score&0.3299&0.2232&0.4380&0.4288&0.6368&0.3101&0.3569\\
			\textbf{(w/ MP NSM)}&size&(1\,380)&(1\,544)&(1\,236)&(111)&(1\,584)&(3\,782)&(2\,452)\\\hline
			EigQ&score&0.2458&0.2294&0.3976&0.4645&0.4340&0.7670&0.4308\\
			\textbf{(w/ MP NSM)}&size&(836)&(1\,459)&(1\,895)&(99)&(3\,367)&(778)&(3\,009)\\\hline\hline
			\multicolumn{2}{l}{$p=22$, eq.\ weights}&&&&&&&\\\hline\hline
			\multirow{2}{*}{NSM}&score&0.6277&0.3421&0.5070&0.6384&0.7239&0.7962&0.7304\\
			&size&(597)&(2\,225)&(2\,075)&(40)&(1\,305)&(539)&(1\,120)\\\hline
			\multirow{2}{*}{ML degree}&score&0.6090&0.3212&0.4755&\second{0.6397}&0.7198&0.7631&0.7000\\
			&size&(623)&(1\,556)&(1\,143)&(46)&(1\,395)&(638)&(1\,148)\\\hline
			\multirow{2}{*}{h-index}&score&0.4480&0.1404&0.0404&0.5045&0.4453&0.2384&0.2528\\
			&size&(374)&(2\,100)&(336)&(83)&(4\,141)&(1\,538)&(2\,206)\\\hline
			\multirow{2}{*}{EigA}&score&0.2487&0.1445&0.3671&0.6234&0.6729&0.2819&0.3663\\
			&size&(1\,870)&(985)&(2\,703)&(46)&(1\,552)&(1\,729)&(2\,712)\\\hline
			\multirow{2}{*}{EigQ}&score&0.4220&0.3077&\first{0.6170}&0.5103&0.5332&0.7919&0.4773\\
			&size&(214)&(2\,384)&(864)&(43)&(1\,665)&(777)&(2\,811)\\\hline\hline
			\multicolumn{2}{l}{$p=2$, opt.\ weights}&&&&&&&\\\hline\hline
			\multirow{2}{*}{\textbf{MP NSM}}&score&\first{0.8563}&\first{0.4834}&\second{0.5462}&\first{0.6463}&\first{0.7629}&\first{0.8292}&\first{0.7651}\\
			&size&(306)&(1\,719)&(1\,279)&(63)&(1\,309)&(553)&(1\,229)\\\hline
			ML degree&score&\second{0.8249}&\second{0.4382}&0.5211&0.5987&\second{0.7369}&0.7946&0.7277\\
			\textbf{(w/ MP NSM)}&size&(332)&(1\,698)&(1\,187)&(43)&(1\,366)&(406)&(1\,039)\\\hline
			h-index&score&0.3197&0.1917&0.1223&0.3820&0.6941&0.2332&0.2124\\
			\textbf{(w/ MP NSM)}&size&(2\,729)&(1\,414)&(1\,899)&(78)&(1\,550)&(1\,538)&(2\,206)\\\hline
			EigA&score&0.5064&0.2580&0.4358&0.5110&0.6788&0.3066&0.3599\\
			\textbf{(w/ MP NSM)}&size&(89)&(1\,581)&(1\,222)&(42)&(1\,599)&(3\,310)&(2\,460)\\\hline
			EigQ&score&0.4539&0.2305&0.3797&0.4430&0.4160&0.7652&0.4280\\
			\textbf{(w/ MP NSM)}&size&(210)&(1\,444)&(2\,057)&(99)&(3\,169)&(778)&(3\,105)\\\hline\hline
			\multicolumn{2}{l}{$p=2$, eq.\ weights}&&&&&&&\\\hline\hline
			\multirow{2}{*}{NSM}&score&0.6291&0.3780&\second{0.5462}&0.6360&0.7246&0.8141&0.7470\\
			&size&(570)&(1\,927)&(1\,427)&(45)&(1\,411)&(541)&(1\,264)\\\hline\hline
		\end{tabular}
	\end{center}
	\caption{Multiplex QUBO score and optimal core size $s^\ast$ for seven real-world multiplex networks.
		Optimised weights correspond to the parameter choice $q=2$ while equal weights refer to the limit case $q \rightarrow \infty$.
		Methods labeled with ``(w/ MP NSM)'' use traditional techniques from the literature combined with the optimised layer weights computed by our method MP NSM.}\label{tab:QUBO_real_world_multiplexes}
\end{table}

\Cref{tab:QUBO_real_world_multiplexes} summarises maximal multiplex QUBO values and optimal core sizes for these networks.
We compare the same core-periphery detection methods considered in \Cref{tab:QUBO_informative_vs_noise_1,tab:QUBO_informative_vs_noise_2} on aggregated versions of the seven multiplex networks using the indicated layer weight vectors $\bm{c}$.
Furthermore, we compare different parameter settings of the multiplex nonlinear spectral method.
In the cases of optimised weights, we choose the parameter $q=2$ while the equal weights case corresponds to the limit case $q \rightarrow \infty$.

\Cref{tab:QUBO_real_world_multiplexes} shows that the maximal multiplex QUBO value is obtained by the multiplex nonlinear spectral method in the parameter setting $p=q=2$ for all but one data set.
However, the second largest value is distributed more homogeneously across different methods and parameter settings.
As described earlier in this subsection, it depends on the distribution of core-periphery structure across the layers whether weights optimised with $q=2$ or equal weights lead to better multiplex core-periphery partitions.

\setlength\tabcolsep{2pt}

\begin{figure}
	\begin{center}
		\begin{tabular}{cccc}
			\hline\hline
			& Layer 1 & Layer 2 & Layer 3\\\hline\hline
			&\multicolumn{3}{c}{$\bm{c} = [0.3439, 0.0241, 0.9387]^T$}\\\hline\hline
			&&&\\[-10pt]
			\parbox[t]{2mm}{\multirow{3}{*}{\rotatebox[origin=c]{90}{\hspace{-40pt} MP NSM}}} & \multirow{2}{104pt}{\centering\includegraphics[width=.235\textwidth,clip,trim=50pt 0pt 50pt 10pt]{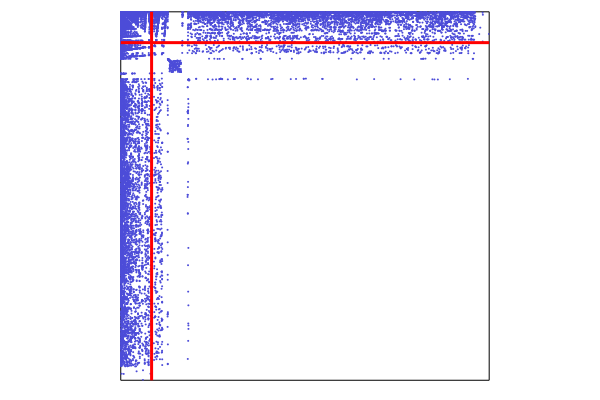}} & \multirow{2}{104pt}{\centering\includegraphics[width=.235\textwidth,clip,trim=50pt 0pt 50pt 10pt]{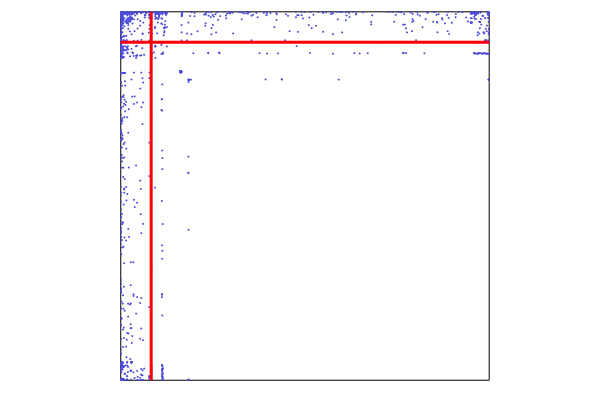}}	& \multirow{2}{104pt}{\centering\includegraphics[width=0.235\textwidth,clip,trim=50pt 0pt 50pt 10pt]{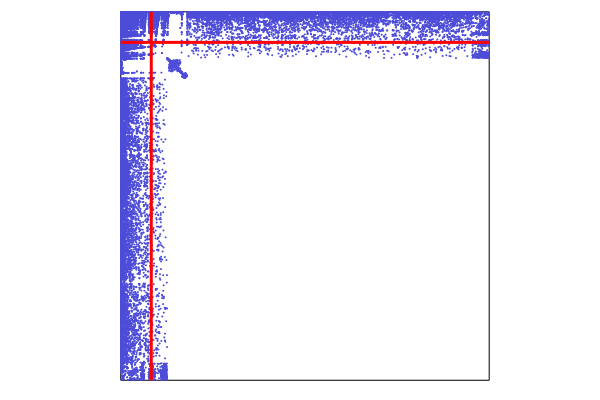}}\\
			&&&\\
			&&&\\
			&&&\\
			&&&\\[-2pt]
			&&&\\\hline
			&&&\\[-10pt]
			\parbox[t]{2mm}{\multirow{3}{*}{\rotatebox[origin=c]{90}{\hspace{-40pt} ML degree opt.}}} & \multirow{2}{104pt}{\centering\includegraphics[width=.235\textwidth,clip,trim=50pt 0pt 50pt 10pt]{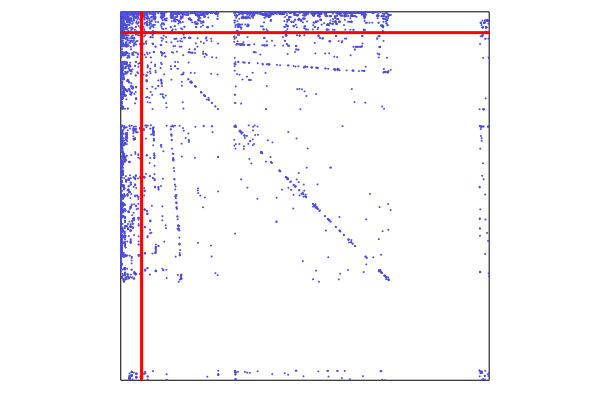}} & \multirow{2}{104pt}{\centering\includegraphics[width=.235\textwidth,clip,trim=50pt 0pt 50pt 10pt]{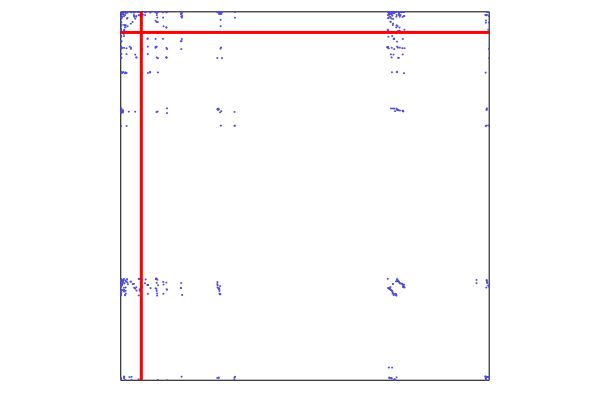}} & \multirow{2}{104pt}{\centering\includegraphics[width=0.235\textwidth,clip,trim=50pt 0pt 50pt 10pt]{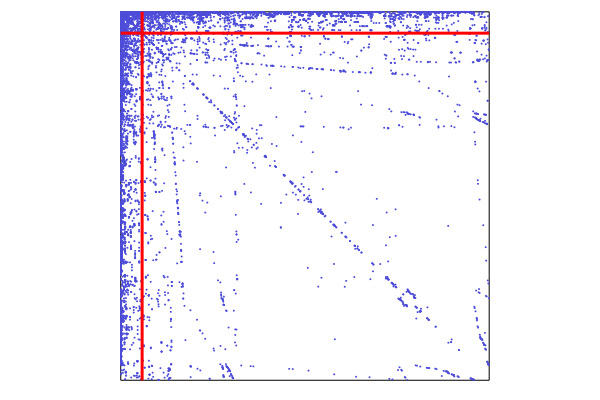}}\\
			&&&\\
			&&&\\
			&&&\\
			&&&\\[-2pt]
			&&&\\\hline\hline
		\end{tabular}
	\end{center}
	\caption{Layer-wise reordered adjacency matrices of the Twitter Rana Plaza $2014$ multiplex network with parameters $\alpha=10$ and $p=2$.
		Optimised layer weights correspond to $q=2$.}\label{fig:rp14_spy}
\end{figure}

Throughout all but one numerical example, the multiplex QUBO values of the multilayer degree method range slightly below those of MP NSM.
\Cref{sec:numerical_experiments_informative_vs_noise} showed that albeit these small differences in multiplex QUBO values, the visual and random walk-based evaluation of multiplex core-periphery structures indicates superiority of the multiplex nonlinear spectral method.
The same observation holds true in the case of real-world multiplex networks.
\Cref{fig:rp14_spy} exemplarily shows that the visual reordering for the Twitter Rana Plaza $2014$ network obtained by MP NSM is much closer to an ideal L-shape than that obtained by ML degree.

Across the real-world multiplex networks considered in this section, we found a varying level of overlap in the cores of the individual layers.
This can be seen by comparing the node coreness vector $\bm{x}$ returned by MP NSM with that obtained by the single-layer NSM \cite{tudisco2019nonlinear}\footnote{which is equivalent to running MP NSM with $L=1$} applied to the individual layers.
We found that nodes belonging to the core in every layer are reliably assigned to the multiplex core while nodes with conflicting assignment across the layers are more likely to be assigned to the multiplex core when belonging to the core in layers with large layer coreness weights.
For instance, in the setting of \Cref{fig:rp14_spy} the single-layer core of layer $1$ consists of $933$ nodes of which $837$ belong to the multiplex core while layer $2$ has $412$ core nodes containing $207$ multiplex core nodes and layer $3$ has $1\,270$ core nodes containing $1165$ out of the $s^\ast=1\,229$ multiplex core nodes.
Clearly, the QUBO objective function values obtained on the individual layers pose upper bounds on the multiplex QUBO score.

The runtime of \Cref{alg} for the computationally most expensive experiment in this section, i.e., the homo multiplex network is $0.69$ seconds for $26$ iterations with the parameters $p=22, q=2$, as well as $3.18$ seconds for $124$ iterations with the parameters $p=2, q=2$.

\section{Conclusion}\label{sec:conclusion}

In this work, we proposed a nonlinear spectral method for multiplex core-periphery detection that simultaneously optimises a node and a layer coreness vector.
We provided convergence guarantees to global and local optima, discussed the detection of optimal core sizes, and illustrated that our method outperforms the state of the art on various synthetic and real-world multiplex networks.

The presented approach can be generalised to multilayer networks with possibly weighted and directed edges within and across layers represented by fourth-order tensors \cite{bergermann2024core}.
Many of the theoretic results presented in this paper carry over to this case after the careful generalisation of objective functions and related quantities.
\vskip6pt

\appendix

\section{Additional numerical experiments}

In this section, we report numerical results supplementing those presented in \Cref{sec:numerical_experiments_informative_vs_noise}.
\Cref{tab:QUBO_informative_vs_noise_2} presents QUBO objective function values for one informative and one noisy layer for the Cardiff and Yeast networks.
Additionally, \Cref{fig:internet_p_22_q_22} shows reordered adjacency matrix plots of the informative layer of the internet network presented in \Cref{fig:internet_p_22_q_2,fig:internet_p_2_q_2} for the parameters $p=q=22$.

\begin{table}
	\small
	\begin{center}
		\begin{tabular}{llcccccc}
			\hline\hline
			\multirow{2}{*}{Network}&&\multicolumn{3}{|c|}{Yeast PPI}&\multicolumn{3}{c|}{Cardiff Tweets}\\
			&&\multicolumn{1}{|c}{$0\%$ noise} & $10\%$ noise & \multicolumn{1}{c|}{$25\%$ noise} & $0\%$ noise & $10\%$ noise & \multicolumn{1}{c|}{$25\%$ noise}\\\hline\hline
			\multicolumn{2}{l}{Layer weight vector $\bm{c}$}&\multirow{2}{*}{$\begin{bmatrix}1\\0\end{bmatrix}$} & \multirow{2}{*}{$\begin{bmatrix}0.9955\\0.0943\end{bmatrix}$} & \multirow{2}{*}{$\begin{bmatrix}0.9723\\0.2339\end{bmatrix}$} & \multirow{2}{*}{$\begin{bmatrix}1\\0\end{bmatrix}$}&\multirow{2}{*}{$\begin{bmatrix}0.9956\\0.0934\end{bmatrix}$}&\multirow{2}{*}{$\begin{bmatrix}0.9741\\0.2259\end{bmatrix}$}\\
			($p=22, q=2$)&&&&&&&\\\hline\hline
			\multirow{2}{*}{\textbf{MP NSM}}&score&\second{0.5310}&\second{0.4870}&\second{0.4300}&\second{0.4581}&\second{0.4177}&\second{0.3799}\\
			&size&(357)&(404)&(386)&(454)&(456)&(483)\\\hline
			ML degree&score&0.5179&0.4753&0.4199&0.4318&0.3964&0.3602\\
			\textbf{(w/ MP NSM)}&size&(357)&(363)&(379)&(376)&(387)&(388)\\\hline
			h-index&score&0.1417&0.1308&0.1135&0.1725&0.1586&0.1364\\
			\textbf{(w/ MP NSM)}&size&(595)&(595)&(597)&(610)&(610)&(615)\\\hline
			EigA&score&0.3617&0.3326&0.2963&0.1025&0.0986&0.1380\\
			\textbf{(w/ MP NSM)}&size&(368)&(410)&(378)&(42)&(702)&(551)\\\hline
			EigQ&score&0.3070&0.2393&0.2345&0.3584&0.3257&0.2747\\
			\textbf{(w/ MP NSM)}&size&(504)&(494)&(253)&(514)&(450)&(514)\\\hline\hline
			\multicolumn{2}{l}{Layer weight vector $\bm{c}$}&\multirow{2}{*}{$\begin{bmatrix}1\\0\end{bmatrix}$} & \multirow{2}{*}{$\begin{bmatrix}1\\1\end{bmatrix}$} & \multirow{2}{*}{$\begin{bmatrix}1\\1\end{bmatrix}$} & \multirow{2}{*}{$\begin{bmatrix}1\\0\end{bmatrix}$} & \multirow{2}{*}{$\begin{bmatrix}1\\1\end{bmatrix}$} & \multirow{2}{*}{$\begin{bmatrix}1\\1\end{bmatrix}$}\\
			\multicolumn{2}{l}{($p=22$, eq.\ weights)}&&&&&&\\\hline\hline
			\multirow{2}{*}{NSM}&score&\second{0.5310}&0.3004&0.2862&\second{0.4581}&0.3052&0.3138\\
			&size&(357)&(424)&(398)&(454)&(563)&(740)\\\hline
			\multirow{2}{*}{ML degree}&score&0.5179&0.2882&0.2787&0.4318&0.2754&0.2865\\
			&size&(357)&(386)&(440)&(376)&(581)&(611)\\\hline
			\multirow{2}{*}{h-index}&score&0.1417&0.0865&0.0693&0.1725&0.0986&0.0778\\
			&size&(595)&(443)&(599)&(610)&(479)&(883)\\\hline
			\multirow{2}{*}{EigA}&score&0.3617&0.1944&0.1933&0.1025&0.0562&0.0934\\
			&size&(368)&(436)&(481)&(42)&(507)&(604)\\\hline
			\multirow{2}{*}{EigQ}&score&0.3070&0.2628&0.1899&0.3584&0.3983&0.2578\\
			&size&(504)&(739)&(742)&(514)&(566)&(656)\\\hline\hline
			\multicolumn{2}{l}{Layer weight vector $\bm{c}$}&\multirow{2}{*}{$\begin{bmatrix}1\\0\end{bmatrix}$} & \multirow{2}{*}{$\begin{bmatrix}0.9996\\0.0277\end{bmatrix}$} & \multirow{2}{*}{$\begin{bmatrix}0.9976\\0.0692\end{bmatrix}$} & \multirow{2}{*}{$\begin{bmatrix}1\\0\end{bmatrix}$}&\multirow{2}{*}{$\begin{bmatrix}0.9997\\0.0228\end{bmatrix}$}&\multirow{2}{*}{$\begin{bmatrix}0.9974\\0.0716\end{bmatrix}$}\\
			\multicolumn{2}{l}{($p=q=2$)}&&&&&&\\\hline\hline
			\multirow{2}{*}{\textbf{MP NSM}}&score&\first{0.5542}&\first{0.5394}&\first{0.5192}&\first{0.4822}&\first{0.4714}&\first{0.4524}\\
			&size&(394)&(388)&(382)&(486)&(486)&(487)\\\hline
			ML degree&score&0.5179&0.5045&0.4847&0.4318&0.4228&0.4069\\
			\textbf{(w/ MP NSM)}&size&(357)&(363)&(363)&(376)&(381)&(388)\\\hline
			h-index&score&0.1417&0.1383&0.1322&0.1725&0.1689&0.1596\\
			\textbf{(w/ MP NSM)}&size&(595)&(595)&(595)&(610)&(610)&(610)\\\hline
			EigA&score&0.3617&0.3525&0.3397&0.1025&0.1000&0.1473\\
			\textbf{(w/ MP NSM)}&size&(368)&(368)&(402)&(42)&(42)&(719)\\\hline
			EigQ&score&0.3070&0.2853&0.2600&0.3584&0.3443&0.3163\\
			\textbf{(w/ MP NSM)}&size&(504)&(374)&(361)&(514)&(524)&(485)\\\hline\hline
			\multicolumn{2}{l}{Layer weight vector $\bm{c}$}&\multirow{2}{*}{$\begin{bmatrix}1\\0\end{bmatrix}$} & \multirow{2}{*}{$\begin{bmatrix}1\\1\end{bmatrix}$} & \multirow{2}{*}{$\begin{bmatrix}1\\1\end{bmatrix}$} & \multirow{2}{*}{$\begin{bmatrix}1\\0\end{bmatrix}$} & \multirow{2}{*}{$\begin{bmatrix}1\\1\end{bmatrix}$} & \multirow{2}{*}{$\begin{bmatrix}1\\1\end{bmatrix}$}\\
			\multicolumn{2}{l}{($p=2$, eq.\ weights)}&&&&&&\\\hline\hline
			\multirow{2}{*}{NSM}&score&\first{0.5542}&0.3595&0.3267&\first{0.4822}&0.3642&0.3612\\
			&size&(394)&(736)&(591)&(486)&(804)&(679)\\\hline\hline
		\end{tabular}
	\end{center}
	\caption{Multiplex QUBO score and optimal core size $s^\ast$ for two two-layer multiplex networks consisting of one real-world single-layer network with strong core-periphery structure and one additional noise layer.
		The noise levels indicate the relative number of non-zero entries in the noise layer in comparison to the informative layer.
		Methods labeled with ``(w/ MP NSM)'' use traditional techniques from the literature combined with the optimised layer weights computed by our method MP NSM.}\label{tab:QUBO_informative_vs_noise_2}
\end{table}

\begin{figure}
	\begin{center}
		\begin{tabular}{ccccc}
			\hline\hline
			& $0\%$ noise & $10\%$ noise & $25\%$ noise & $50\%$ noise\\\hline\hline
			& $\bm{c}=[1,0]^T$ & $\bm{c} = [0.9966, 0.8870]^T$ & $\bm{c} = [0.9918, 0.9215]^T$ & $\bm{c} = [0.9843, 0.9459]^T$\\\hline\hline
			&&&&\\[-10pt]
			\parbox[t]{2mm}{\multirow{3}{*}{\rotatebox[origin=c]{90}{\hspace{-40pt} MP NSM}}} & \multirow{2}{85pt}{\centering\includegraphics[width=.235\textwidth,clip,trim=50pt 0pt 50pt 10pt]{graphics/NSM_internet_single-layer_p_22.png}} & \multirow{2}{85pt}{\centering\includegraphics[width=.235\textwidth,clip,trim=50pt 0pt 50pt 10pt]{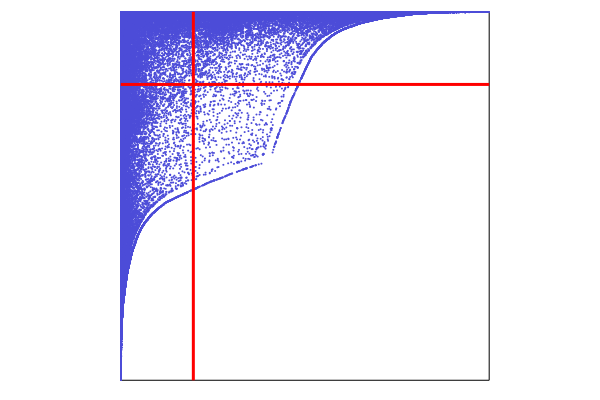}} & \multirow{2}{85pt}{\centering\includegraphics[width=0.235\textwidth,clip,trim=50pt 0pt 50pt 10pt]{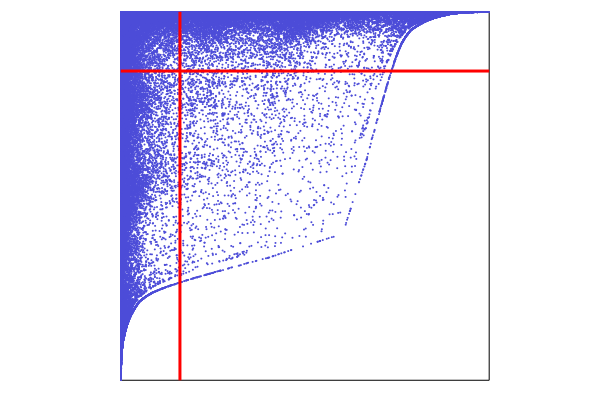}} & \multirow{2}{85pt}{\centering\includegraphics[width=0.235\textwidth,clip,trim=50pt 0pt 50pt 10pt]{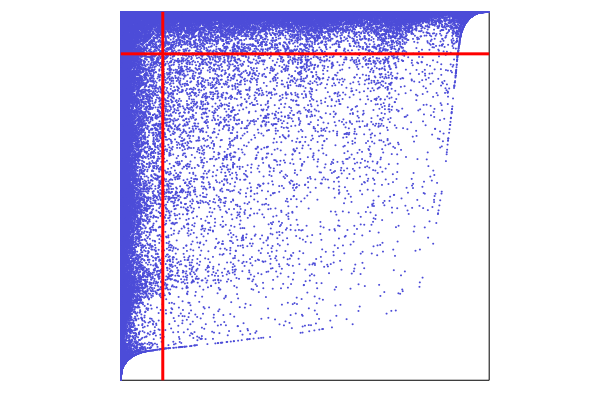}}\\
			&&&&\\
			&&&&\\
			&&&&\\
			&&&&\\[-2pt]
			&&&&\\\hline
			&&&&\\[-10pt]
			\parbox[t]{2mm}{\multirow{3}{*}{\rotatebox[origin=c]{90}{\hspace{-40pt} ML degree opt.}}} & \multirow{2}{85pt}{\centering\includegraphics[width=.235\textwidth,clip,trim=50pt 0pt 50pt 10pt]{graphics/MLdegree_internet_single-layer.png}} & \multirow{2}{85pt}{\centering\includegraphics[width=.235\textwidth,clip,trim=50pt 0pt 50pt 10pt]{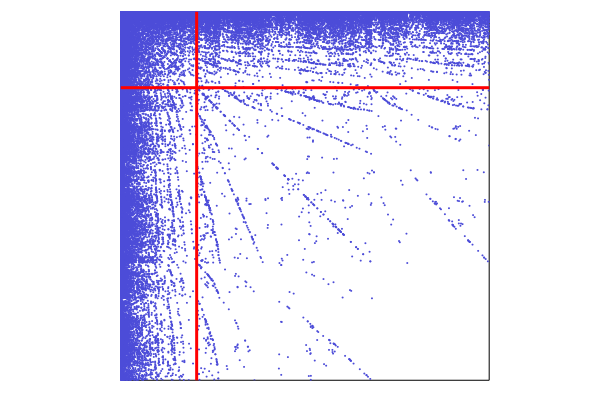}} & \multirow{2}{85pt}{\centering\includegraphics[width=0.235\textwidth,clip,trim=50pt 0pt 50pt 10pt]{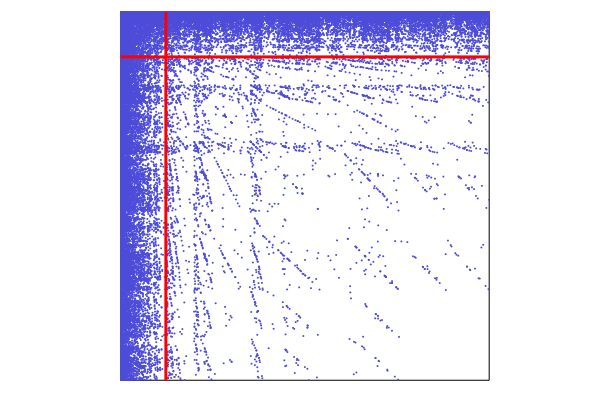}} & \multirow{2}{85pt}{\centering\includegraphics[width=0.235\textwidth,clip,trim=50pt 0pt 50pt 10pt]{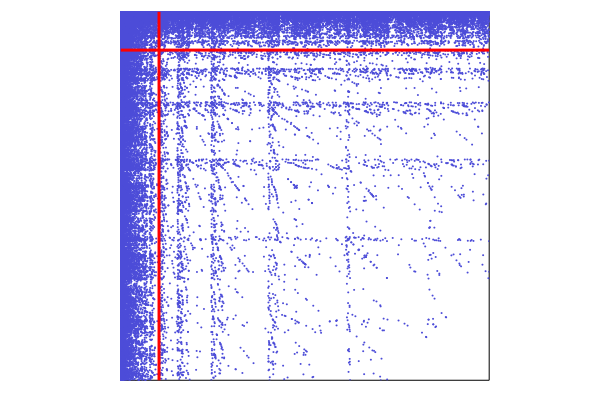}}\\
			&&&&\\
			&&&&\\
			&&&&\\
			&&&&\\[-2pt]
			&&&&\\\hline\hline
		\end{tabular}
	\end{center}
	\caption{Reordered adjacency matrices of the informative layer of the two-layer Internet multiplex network for the parameters $\alpha=10$ and $p=q=22$ for various levels of noise in the second uninformative layer and optimised layer weights.
	The results for equal layer weights coincide with those displayed in \Cref{fig:internet_p_22_q_2}.}\label{fig:internet_p_22_q_22}
\end{figure}

\section{Proof of \Cref{thm:unique_solution}}\label{sec:appendix_proof_unique_solution}

We require the following two Lemmas before we turn to the proof of \Cref{thm:unique_solution}.

\begin{lemma}\label{lemma:coefficient_matrix}
	The objective function $f_\alpha(\bm{x},\bm{c})$ defined in \cref{eq:objective} satisfies the elementwise inequalities
	\begin{equation}\label{eq:theta_multilayer_setting}
	\begin{array}{cc}
	|\nabla_{\bm{x}} \nabla_{\bm{x}} f_\alpha (\bm{x},\bm{c})\bm{x}| \leq \bm{\Theta}_{11} |\nabla_{\bm{x}} f_\alpha(\bm{x},\bm{c})|
	&  
	|\nabla_{\bm{c}} \nabla_{\bm{x}} f_\alpha (\bm{x},\bm{c})\bm{c}| \leq \bm{\Theta}_{12} |\nabla_{\bm{x}} f_\alpha(\bm{x},\bm{c})|
	\\
	|\nabla_{\bm{x}} \nabla_{\bm{c}} f_\alpha (\bm{x},\bm{c})\bm{x}| \leq \bm{\Theta}_{21} |\nabla_{\bm{c}} f_\alpha(\bm{x},\bm{c})|
	&
	|\nabla_{\bm{c}} \nabla_{\bm{c}} f_\alpha (\bm{x},\bm{c})\bm{c}| \leq \bm{\Theta}_{22} |\nabla_{\bm{c}} f_\alpha(\bm{x},\bm{c})|
	\end{array}
	\end{equation}
	with the coefficient matrix
	\begin{equation*}
	\bm{\Theta} = \begin{bmatrix}
	\bm{\Theta}_{11} & \bm{\Theta}_{12}\\
	\bm{\Theta}_{21} & \bm{\Theta}_{22}
	\end{bmatrix}
	=
	\begin{bmatrix}
	|\alpha-1| & 1\\
	2 & 0
	\end{bmatrix}.
	\end{equation*}
\end{lemma}
\begin{proof}
	
	Throughout the proof, sums over $i$ and $j$ include indices $i,j=1,\dots,n$ while sums over $k$ include indices $k=1,\dots,L$.
	The entries of the gradients of \cref{eq:objective} read
	\begin{equation*}
	\left[ \nabla_{\bm{x}} f_\alpha(\bm{x},\bm{c}) \right]_m = 2 \sum_j \bm{x}_m^{\alpha-1} (\bm{x}_m^\alpha + \bm{x}_j^\alpha)^{1/\alpha-1} \sum_k \bm{c}_k \bm{A}_{mj}^{(k)}, \quad 
	\left[ \nabla_{\bm{c}} f_\alpha(\bm{x},\bm{c}) \right]_l = \sum_{i,j} \bm{A}_{ij}^{(l)} (\bm{x}_i^\alpha + \bm{x}_j^\alpha)^{1/\alpha}.
	\end{equation*}
	
	Furthermore, the entries of the Hessian of \cref{eq:objective} read
	\begin{align*}
	[\nabla_{\bm{x}}\nabla_{\bm{x}} f_\alpha(\bm{x},\bm{c})]_{mi} & = 2 (1-\alpha) \bm{x}_m^{\alpha-1} \bm{x}_i^{\alpha-1} (\bm{x}_m^\alpha + \bm{x}_i^\alpha)^{1/\alpha-2} \sum_k \bm{c}_k \bm{A}_{mi}^{(k)}, \quad \text{ for } m\neq i,\\
	[\nabla_{\bm{x}}\nabla_{\bm{x}} f_\alpha(\bm{x},\bm{c})]_{mm} & = 2 (\alpha-1) \bm{x}_m^{\alpha-2} \sum_{\substack{j\\j\neq m}} \bm{x}_j^\alpha (\bm{x}_m^\alpha + \bm{x}_j^\alpha)^{1/\alpha-2} \sum_k \bm{c}_k \bm{A}_{mj}^{(k)},\\
	[\nabla_{\bm{x}} \nabla_{\bm{c}} f_\alpha(\bm{x},\bm{c})]_{ml} & = [\nabla_{\bm{c}} \nabla_{\bm{x}} f_\alpha(\bm{x},\bm{c})]_{lm} = 2 \sum_j \bm{A}_{mj}^{(l)} \bm{x}_m^{\alpha-1} (\bm{x}_m^\alpha + \bm{x}_j^\alpha)^{1/\alpha-1},\\
	\nabla_{\bm{c}}\nabla_{\bm{c}} f_\alpha(\bm{x},\bm{c}) & = \bm{0}.
	\end{align*}
	\begin{sloppypar}
		Note that since $\alpha>1$ by assumption, all gradient and Hessian entries except for $[\nabla_{\bm{x}}\nabla_{\bm{x}} f_\alpha(\bm{x},\bm{c})]_{mi}$ for $m\neq i$ are non-negative.
	\end{sloppypar}
	
	Since $\nabla_{\bm{c}}\nabla_{\bm{c}} f_\alpha(\bm{x},\bm{c}) = \bm{0}$, we have $\bm{\Theta}_{22}=0$ and by
	\begin{equation*}
	[\nabla_{\bm{c}} \nabla_{\bm{x}} f_\alpha(\bm{x},\bm{c}) \bm{c}]_m = 2 \sum_{j} \bm{x}_m^{\alpha-1}(\bm{x}_m^\alpha + \bm{x}_j^\alpha)^{1/\alpha-1} \sum_k \bm{c}_k \bm{A}_{mj}^{(k)} = \left[ \nabla_{\bm{x}} f_\alpha(\bm{x},\bm{c}) \right]_m,
	\end{equation*}
	we immediately see $\bm{\Theta}_{12} = 1$.
	
	To obtain $\bm{\Theta}_{21}=2$, consider
	\begin{align*}
	[\nabla_{\bm{x}} \nabla_{\bm{c}} f_\alpha(\bm{x},\bm{c})\bm{x}]_l & = 2 \sum_{i,j} \bm{A}_{ij}^{(l)} \bm{x}_i^\alpha (\bm{x}_i^\alpha + \bm{x}_j^\alpha)^{1/\alpha-1}\\
	& \leq  2 \sum_{i,j} \bm{A}_{ij}^{(l)} (\bm{x}_i^\alpha + \underbrace{\bm{x}_j^\alpha}_{\geq 0}) (\bm{x}_i^\alpha + \bm{x}_j^\alpha)^{1/\alpha-1} = 2 [\nabla_{\bm{c}} f_\alpha(\bm{x},\bm{c})]_l
	\end{align*}
	for all $l=1, \dots , L$.
	
	Finally, for $\bm{\Theta}_{11}$, we consider the cases $i\neq m$ and $i=m$ separately. For $i\neq m$, we have
	\begin{align*}
	\frac{[\nabla_{\bm{x}}\nabla_{\bm{x}} f_\alpha(\bm{x},\bm{c})]_{mi}\bm{x}_i}{[\nabla_{\bm{x}} f_\alpha(\bm{x},\bm{c})]_m} & = \frac{2(1-\alpha)\bm{x}_m^{\alpha-1} \bm{x}_i^\alpha (\bm{x}_m^\alpha + \bm{x}_i^\alpha)^{1/\alpha-2} \sum_k \bm{c}_k \bm{A}_{mi}^{(k)}}{2 \sum_j \bm{x}_m^{\alpha-1} (\bm{x}_m^\alpha + \bm{x}_j^\alpha)^{1/\alpha-1} \sum_k \bm{c}_k \bm{A}_{mj}^{(k)}}\\
	& = (1-\alpha) \frac{\bm{x}_i^\alpha}{\bm{x}_m^\alpha + \bm{x}_i^\alpha} \frac{(\bm{x}_m^\alpha + \bm{x}_i^\alpha)^{1/\alpha-1} \sum_k \bm{c}_k \bm{A}_{mi}^{(k)}}{\sum_j (\bm{x}_m^\alpha + \bm{x}_j^\alpha)^{1/\alpha-1} \sum_k \bm{c}_k \bm{A}_{mj}^{(k)}}\\
	& \geq (1-\alpha) \frac{(\bm{x}_m^\alpha + \bm{x}_i^\alpha)^{1/\alpha-1} \sum_k \bm{c}_k \bm{A}_{mi}^{(k)}}{\sum_j (\bm{x}_m^\alpha + \bm{x}_j^\alpha)^{1/\alpha-1} \sum_k \bm{c}_k \bm{A}_{mj}^{(k)}},
	\end{align*}
	since, by assumption, we have $\alpha>1 \Leftrightarrow 1-\alpha < 0$.
	
	Summing over $i=\{1, \dots , n\} \setminus \{m\}$ yields
	\begin{align*}
	\frac{\sum_{\substack{i\\i\neq m}} [\nabla_{\bm{x}}\nabla_{\bm{x}} f_\alpha(\bm{x},\bm{c})]_{mi} \bm{x}_i}{[\nabla_{\bm{x}} f_\alpha(\bm{x},\bm{c})]_m} & = (1-\alpha) \frac{\sum_{\substack{i\\i\neq m}} (\bm{x}_m^\alpha + \bm{x}_i^\alpha)^{1/\alpha-1} \sum_k \bm{c}_k \bm{A}_{mi}^{(k)}}{\sum_j (\bm{x}_m^\alpha + \bm{x}_j^\alpha)^{1/\alpha-1} \sum_k \bm{c}_k \bm{A}_{mj}^{(k)}}\\
	& \geq (1-\alpha).
	\end{align*}
	\begin{sloppypar}
		Moreover, since all summands of $\sum_{\substack{i\\i\neq m}} (\bm{x}_m^\alpha + \bm{x}_i^\alpha)^{1/\alpha-1} \sum_k \bm{c}_k \bm{A}_{mi}^{(k)}$ and \mbox{$\sum_j (\bm{x}_m^\alpha + \bm{x}_j^\alpha)^{1/\alpha-1} \sum_k \bm{c}_k \bm{A}_{mj}^{(k)}$} are non-negative, we have
		\begin{equation*}
		\frac{\sum_{\substack{i\\i\neq m}} [\nabla_{\bm{x}}\nabla_{\bm{x}} f_\alpha(\bm{x},\bm{c})]_{mi} \bm{x}_i}{[\nabla_{\bm{x}} f_\alpha(\bm{x},\bm{c})]_m} \leq 0.
		\end{equation*}
	\end{sloppypar}
	
	For $i=m$, we have
	\begin{align*}
	\frac{[\nabla_{\bm{x}}\nabla_{\bm{x}} f_\alpha(\bm{x},\bm{c})]_{mm}\bm{x}_m}{[\nabla_{\bm{x}} f_\alpha(\bm{x},\bm{c})]_m} & = \frac{2(\alpha-1)\bm{x}_m^{\alpha-1} \sum_{\substack{j\\j\neq m}} \bm{x}_j^\alpha (\bm{x}_m^\alpha + \bm{x}_j^\alpha)^{1/\alpha-2} \sum_k \bm{c}_k \bm{A}_{mj}^{(k)}}{2 \sum_j  \bm{x}_m^{\alpha-1} (\bm{x}_m^\alpha + \bm{x}_j^\alpha)^{1/\alpha-1} \sum_k \bm{c}_k \bm{A}_{mj}^{(k)}}\\
	& = (\alpha-1) \frac{\sum_{\substack{j\\j\neq m}} \frac{\bm{x}_j^\alpha}{\bm{x}_m^\alpha + \bm{x}_j^\alpha} (\bm{x}_m^\alpha + \bm{x}_j^\alpha)^{1/\alpha-1} \sum_k \bm{c}_k \bm{A}_{mj}^{(k)}}{\sum_j (\bm{x}_m^\alpha + \bm{x}_j^\alpha)^{1/\alpha-1} \sum_k \bm{c}_k \bm{A}_{mj}^{(k)}}\\
	& \leq (\alpha-1) \frac{\sum_{\substack{j\\j\neq m}} (\bm{x}_m^\alpha + \bm{x}_j^\alpha)^{1/\alpha-1} \sum_k \bm{c}_k \bm{A}_{mj}^{(k)}}{\sum_j (\bm{x}_m^\alpha + \bm{x}_j^\alpha)^{1/\alpha-1} \sum_k \bm{c}_k \bm{A}_{mj}^{(k)}}\\
	& \leq (\alpha-1).
	\end{align*}
	At the same time, we have $\frac{[\nabla_{\bm{x}}\nabla_{\bm{x}} f_\alpha(\bm{x},\bm{c})]_{mm}\bm{x}_m}{[\nabla_{\bm{x}} f_\alpha(\bm{x},\bm{c})]_m}\geq 0$ since all elements are non-negative.
	
	Finally, we obtain
	\begin{align*}
	& (1-\alpha) \leq \sum_{\substack{i\\i\neq m}} [\nabla_{\bm{x}}\nabla_{\bm{x}} f_\alpha(\bm{x},\bm{c})]_{mi} \bm{x}_i + [\nabla_{\bm{x}}\nabla_{\bm{x}} f_\alpha(\bm{x},\bm{c})]_{mm} \bm{x}_m \leq (\alpha-1)\\
	\Leftrightarrow &~\frac{\left| \sum_{\substack{i\\i\neq m}} [\nabla_{\bm{x}}\nabla_{\bm{x}} f_\alpha(\bm{x},\bm{c})]_{mi} \bm{x}_i + [\nabla_{\bm{x}}\nabla_{\bm{x}} f_\alpha(\bm{x},\bm{c})]_{mm} \bm{x}_m\right|}{|[\nabla_{\bm{x}} f_\alpha(\bm{x},\bm{c})]_m|} \leq |\alpha-1| \color{black} = \bm{\Theta}_{11}.
	\end{align*}
\end{proof}

Again using the notation $J_p(\bm{x}) := \nabla \|\bm{x}\|_p = \|\bm{x}\|_p^{1-p}\bm{x}^{p-1}$, we define the fixed point map 
\begin{equation}\label{eq:fixed_point_map}
G^\alpha(\bm{x},\bm{c}) := \begin{bmatrix}
\nabla_{\bm{x}} g_\alpha(\bm{x},\bm{c})\\
\nabla_{\bm{c}} g_\alpha(\bm{x},\bm{c})
\end{bmatrix} = \begin{bmatrix}
J_{p^\ast}(\nabla_{\bm{x}} f_\alpha(\bm{x},\bm{c}))\\
J_{q^\ast}(\nabla_{\bm{c}} f_\alpha(\bm{x},\bm{c}))
\end{bmatrix}
:=
\begin{bmatrix}
G^\alpha_1(\bm{x},\bm{c})\\
G^\alpha_2(\bm{x},\bm{c})
\end{bmatrix},
\end{equation}
corresponding to the iteration \cref{eq:iteration_multilayer}.
We denote by $D_{\bm{x}} G^\alpha_i(\bm{x},\bm{c})$ and $D_{\bm{c}} G^\alpha_i(\bm{x},\bm{c})$ for $i=1,2$ the Jacobians w.r.t.\ $\bm{x}$ and $\bm{c}$, respectively.
Consequently, quantities such as $D_{\bm{x}} G^\alpha_1(\bm{x},\bm{c})\bm{x}$ denote matrix-vector products.

\begin{lemma}\label{lemma:fixed_point_map_bounds}
	For elementwise divisions, the map $G^\alpha$ defined in \cref{eq:fixed_point_map} satisfies the element\-wise inequality
	\begin{equation}\label{eq:coeff_matrix}
	\begin{bmatrix}
	\left\|\frac{D_{\bm{x}}G^\alpha_1(\bm{x},\bm{c})\bm{x}}{G^\alpha_1(\bm{x},\bm{c})}\right\|_1 & \left\|\frac{D_{\bm{c}}G^\alpha_1(\bm{x},\bm{c})\bm{c}}{G^\alpha_1(\bm{x},\bm{c})}\right\|_1\\[.7em]
	\left\|\frac{D_{\bm{x}}G^\alpha_2(\bm{x},\bm{c})\bm{x}}{G^\alpha_2(\bm{x},\bm{c})}\right\|_1 & \left\|\frac{D_{\bm{c}}G^\alpha_2(\bm{x},\bm{c})\bm{c}}{G^\alpha_2(\bm{x},\bm{c})}\right\|_1
	\end{bmatrix}
	\leq
	\begin{bmatrix}
	\frac{2|\alpha-1|}{p-1} & \frac{1}{p-1}\\[.7em]
	\frac{2}{q-1} & 0
	\end{bmatrix}
	=: \bm{M}.
	\end{equation}
\end{lemma}
\begin{proof}
	
	We define $F_{\bm{x}}^\alpha:=\nabla_{\bm{x}} f_\alpha(\bm{x},\bm{c}), F_{\bm{c}}^\alpha:=\nabla_{\bm{c}} f_\alpha(\bm{x},\bm{c}), \partial_r:=\frac{\partial}{\partial \bm{x}_r}, \partial_s:=\frac{\partial}{\partial \bm{c}_s}$ and throughout this proof, sums over $i$ include indices $i=1, \dots  n$ while sums over $k$ include indices $k=1, \dots , L$.
	By the product and the chain rule, we obtain
	\begin{align}
	\frac{\partial_r \left[ G_1^\alpha(\bm{x},\bm{c}) \right]_m}{\left[ G_1^\alpha(\bm{x},\bm{c}) \right]_m} & = \frac{\partial_r \left( \| F_{\bm{x}}^\alpha \|_{p^\ast}^{1-p^\ast} \left( \left[ F_{\bm{x}}^\alpha \right]_m \right)^{p^\ast-1} \right)}{\| F_{\bm{x}}^\alpha \|_{p^\ast}^{1-p^\ast} \left( \left[ F_{\bm{x}}^\alpha \right]_m \right)^{p^\ast-1}} = (p^\ast-1) \left\{ \frac{\partial_r \left[ F_{\bm{x}}^\alpha \right]_m}{\left[ F_{\bm{x}}^\alpha \right]_m} - \frac{\partial_r \| F_{\bm{x}}^\alpha \|_{p^\ast}}{\| F_{\bm{x}}^\alpha \|_{p^\ast}}\right\}\nonumber\\
	& = (p^\ast-1) \left\{ \frac{\partial_r \left[ F_{\bm{x}}^\alpha \right]_m}{\left[ F_{\bm{x}}^\alpha \right]_m} - \frac{\sum_i \left[ F_{\bm{x}}^\alpha \right]_i^{p^\ast-1} \partial_r \left[ F_{\bm{x}}^\alpha \right]_i}{\| F_{\bm{x}}^\alpha \|_{p^\ast}^{p^\ast}}\right\}\nonumber\\
	& = (p^\ast-1) \left\{ \frac{\partial_r \left[ F_{\bm{x}}^\alpha \right]_m}{\left[ F_{\bm{x}}^\alpha \right]_m} - \sum_i \left( \frac{\left[ F_{\bm{x}}^\alpha \right]_i}{\| F_{\bm{x}}^\alpha \|_{p^\ast}} \right)^{p^\ast} \frac{\partial_r \left[ F_{\bm{x}}^\alpha \right]_i}{\left[ F_{\bm{x}}^\alpha \right]_i}\right\}\label{eq:lemma2_xx_deriv}
	\end{align}
	
	Analogous computations yield
	\begin{align}
	\frac{\partial_s \left[ G_1^\alpha(\bm{x},\bm{c}) \right]_m}{\left[ G_1^\alpha(\bm{x},\bm{c}) \right]_m} & = (p^\ast-1) \left\{ \frac{\partial_s \left[ F_{\bm{x}}^\alpha \right]_m}{\left[ F_{\bm{x}}^\alpha \right]_m} - \sum_i \left( \frac{\left[ F_{\bm{x}}^\alpha \right]_i}{\| F_{\bm{x}}^\alpha \|_{p^\ast}} \right)^{p^\ast} \frac{\partial_s \left[ F_{\bm{x}}^\alpha \right]_i}{\left[ F_{\bm{x}}^\alpha \right]_i}\right\},\label{eq:lemma2_xc_deriv}\\
	\frac{\partial_r \left[ G_2^\alpha(\bm{x},\bm{c}) \right]_l}{\left[ G_2^\alpha(\bm{x},\bm{c}) \right]_l} & = (q^\ast-1) \left\{ \frac{\partial_r \left[ F_{\bm{c}}^\alpha \right]_l}{\left[ F_{\bm{c}}^\alpha \right]_l} - \sum_k \left( \frac{\left[ F_{\bm{c}}^\alpha \right]_k}{\| F_{\bm{c}}^\alpha \|_{q^\ast}} \right)^{q^\ast} \frac{\partial_r \left[ F_{\bm{c}}^\alpha \right]_k}{\left[ F_{\bm{c}}^\alpha \right]_k}\right\},\label{eq:lemma2_cx_deriv}\\
	\frac{\partial_s \left[ G_2^\alpha(\bm{x},\bm{c}) \right]_l}{\left[ G_2^\alpha(\bm{x},\bm{c}) \right]_l} & = (q^\ast-1) \left\{ \frac{\partial_s \left[ F_{\bm{c}}^\alpha \right]_l}{\left[ F_{\bm{c}}^\alpha \right]_l} - \sum_k \left( \frac{\left[ F_{\bm{c}}^\alpha \right]_k}{\| F_{\bm{c}}^\alpha \|_{q^\ast}} \right)^{q^\ast} \frac{\partial_s \left[ F_{\bm{c}}^\alpha \right]_k}{\left[ F_{\bm{c}}^\alpha \right]_k}\right\}.\label{eq:lemma2_cc_deriv}
	\end{align}
	
	Taking the absolute value of \eqref{eq:lemma2_xx_deriv} leads to a quantity of the form $|\beta_m - \sum_i \gamma_i\beta_i|$, where $\gamma_i\geq 0$, $\sum_i \gamma_i = 1$, and $\beta_m$ may be positive or negative.
	Defining $\widetilde{m}:=\text{argmax}_m |\beta_m|$, we have $|\beta_m - \sum_i \gamma_i\beta_i| \leq 2|\beta_{\widetilde{m}}|$.
	
	The absolute value of \eqref{eq:lemma2_xc_deriv} takes the same form, but here all $\beta_m$ are non-negative.
	Hence, we have $|\beta_m - \sum_i \gamma_i\beta_i| \leq |\beta_{\widetilde{m}}|$.
	
	Absolute values of both \eqref{eq:lemma2_cx_deriv} and \eqref{eq:lemma2_cc_deriv} yield $|\beta_l - \sum_k \gamma_k\beta_k|$ with $\gamma_k\geq 0$, $\sum_k \gamma_k = 1$, and $\beta_l$ non-negative such that we define $\widetilde{l} := \text{argmax}_l |\beta_l|$ to find $|\beta_l - \sum_k \gamma_k\beta_k| \leq |\beta_{\widetilde{l}}|$ in both cases.
	
	Then, multiplication with $\bm{x}_m\geq 0$ and $\bm{c}_l\geq 0$ as well as summation over $m$ and $l$, respectively, the bounds from \Cref{lemma:coefficient_matrix}, as well as $p^\ast-1=\frac{1}{p-1}$ and $q^\ast-1=\frac{1}{q-1}$ yield the claim.
\end{proof}

With these results at hand, we can prove the global convergence of \cref{eq:iteration_multilayer}, i.e., \Cref{alg} under the conditions stated in \Cref{thm:unique_solution}.

\begin{proof}[Proof of \Cref{thm:unique_solution}]
	\begin{sloppypar}
		We show that $G^\alpha$ defined in \cref{eq:fixed_point_map} is contractive with respect to a suitable metric with Lipschitz constant $\rho(\bm{M})$.
		We start by recalling the Thompson metric $d_T(\bm{x},\bm{y}) = \|\log(\bm{x}) - \log(\bm{y})\|_\infty$ for $\bm{x},\bm{y}\in\R^n_{>0}$.
		Throughout this proof, the functions $\log(\bm{x})$ and $\exp(\bm{x})$ are applied elementwise.
		With this, we define $\widetilde{d_T}([\bm{x},\bm{c}]^T,[\bm{y},\bm{d}]^T) := \bm{v}_1 d_T(\bm{x},\bm{y}) + \bm{v}_2 d_T(\bm{c},\bm{d})$ for $\bm{c}, \bm{d}\in\R^L_{>0}$ and some $\bm{v}\in\R^2_{>0}$.
		Note that $(\mathcal{M},\widetilde{d_T})$ with $\mathcal{M} = \{ [\bm{x},\bm{c}]^T : \bm{x}\in\mathcal{S}_p^+, \bm{c}\in\mathcal{S}_q^+ \}$ represents a complete metric space for any such $\bm{v}$ \cite{gautier2019perron}.
		
		Furthermore, we define the maps $\phi_1:\R^n \times \R^L \rightarrow \R^n$ and $\phi_2:\R^n \times \R^L \rightarrow \R^L$ via $\phi_i(\bm{x},\bm{c}) := \log(G_i^\alpha(\exp(\bm{x}),\exp(\bm{c})))$ for $i=1,2$ as well as the differentials $\widetilde{D_{\bm{x}}}$ and $\widetilde{D_{\bm{c}}}$, which return the first $n$ and last $L$ entries of $D_{\bm{x}}$ and $D_{\bm{c}}$, respectively.
		Then, by the H\"older inequality \mbox{$\| \phi(\bm{x}) - \phi(\bm{y}) \|_\infty \leq \sup_{\bm{\xi}\in[\bm{x},\bm{y}]} \| D \phi(\bm{\xi}) \|_1 \| \bm{x} - \bm{y} \|_\infty$}, we have for $\widetilde{\bm{x}}=\exp(\bm{x})$, $\widetilde{\bm{y}}=\exp(\bm{y})$, $\widetilde{\bm{c}}=\exp(\bm{c})$, and $\widetilde{\bm{d}}=\exp(\bm{d})$ that
		\begin{align*}
		d_T(G_1^\alpha(\widetilde{\bm{x}},\widetilde{\bm{c}}),G_1^\alpha(\widetilde{\bm{y}},\widetilde{\bm{d}})) = & \| \phi_1(\bm{x},\bm{c}) - \phi_1(\bm{y},\bm{d}) \|_\infty\\
		\leq & \sup_{\bm{\xi}_1\in[\bm{x},\bm{y}]} \| \widetilde{D_{\bm{x}}} \phi_{1}(\bm{\xi}_1) \|_1 \| \log(\widetilde{\bm{x}}) - \log(\widetilde{\bm{y}}) \|_\infty\\
		& + \sup_{\bm{\xi}_2\in[\bm{c},\bm{d}]} \| \widetilde{D_{\bm{c}}} \phi_{1}(\bm{\xi}_2) \|_1 \| \log(\widetilde{\bm{c}}) - \log(\widetilde{\bm{d}}) \|_\infty\\
		= & \sup_{\bm{\xi}_1\in[\bm{x},\bm{y}]} \| \widetilde{D_{\bm{x}}} \phi_{1}(\bm{\xi}_1) \|_1 d_T(\widetilde{\bm{x}},\widetilde{\bm{y}}) + \sup_{\bm{\xi}_2\in[\bm{c},\bm{d}]} \| \widetilde{D_{\bm{c}}} \phi_{1}(\bm{\xi}_2) \|_1 d_T(\widetilde{\bm{c}},\widetilde{\bm{d}}),
		\end{align*}
		where $\bm{\xi}_1\in[\bm{x},\bm{y}]$ and $\bm{\xi}_2\in[\bm{c},\bm{d}]$ denote elements on the line segments between $\bm{x}$ and $\bm{y}$ as well as $\bm{c}$ and $\bm{d}$, respectively.
		Additionally, the analogous inequality holds for $G_{2}^\alpha$ and $\phi_2$.
	\end{sloppypar}
	
	Since $\widetilde{D_{\bm{x}}} \phi_{i}(\bm{x},\bm{c}) = \frac{D_{\bm{x}} G_i^\alpha(\widetilde{\bm{x}},\widetilde{\bm{c}}) \widetilde{\bm{x}}}{G_i^\alpha(\widetilde{\bm{x}},\widetilde{\bm{c}})}$ and $\widetilde{D_{\bm{c}}} \phi_{i}(\bm{x},\bm{c}) = \frac{D_{\bm{c}} G_i^\alpha(\widetilde{\bm{x}},\widetilde{\bm{c}}) \widetilde{\bm{c}}}{G_i^\alpha(\widetilde{\bm{x}},\widetilde{\bm{c}})}$ for $i=1,2$, we have
	\begin{equation*}
	\begin{bmatrix}
	\sup_{\bm{\xi}_1\in[\bm{x},\bm{y}]} \| \widetilde{D_{\bm{x}}} \phi_{1}(\bm{\xi}_1) \|_1 & \sup_{\bm{\xi}_2\in[\bm{c},\bm{d}]} \| \widetilde{D_{\bm{c}}} \phi_{1}(\bm{\xi}_2) \|_1\\
	\sup_{\bm{\xi}_1\in[\bm{x},\bm{y}]} \| \widetilde{D_{\bm{x}}} \phi_{2}(\bm{\xi}_1) \|_1 & \sup_{\bm{\xi}_2\in[\bm{c},\bm{d}]} \| \widetilde{D_{\bm{c}}} \phi_{2}(\bm{\xi}_2) \|_1
	\end{bmatrix}
	=:
	\widetilde{\bm{M}},
	\end{equation*}
	with $\widetilde{\bm{M}} \leq \bm{M}$ elementwise where $\bm{M}$ is defined in \Cref{lemma:fixed_point_map_bounds}.
	Note that $\bm{M}$ is non-negative and irreducible, so there exists a $\bm{v}\in\R_{>0}^2$ such that $\bm{v}^T\bm{M} = \rho(\bm{M})\bm{v}^T$.
	Choosing such $\bm{v}$, we have
	\begin{align*}
	\widetilde{d_T}(G^\alpha(\widetilde{\bm{x}},\widetilde{\bm{c}}),G^\alpha(\widetilde{\bm{y}},\widetilde{\bm{d}})) & = \bm{v}^T
	\begin{bmatrix}
	d_T(G^\alpha_1(\widetilde{\bm{x}},\widetilde{\bm{c}}),G^\alpha_1(\widetilde{\bm{y}},\widetilde{\bm{d}}))\\
	d_T(G^\alpha_2(\widetilde{\bm{x}},\widetilde{\bm{c}}),G^\alpha_2(\widetilde{\bm{y}},\widetilde{\bm{d}}))
	\end{bmatrix}
	\leq
	\bm{v}^T \bm{M}
	\begin{bmatrix}
	d_T(\widetilde{\bm{x}},\widetilde{\bm{y}})\\
	d_T(\widetilde{\bm{c}},\widetilde{\bm{d}})
	\end{bmatrix}\\
	& = \rho(\bm{M}) \bm{v}^T
	\begin{bmatrix}
	d_T(\widetilde{\bm{x}},\widetilde{\bm{y}})\\
	d_T(\widetilde{\bm{c}},\widetilde{\bm{d}})
	\end{bmatrix}
	= \rho(\bm{M}) \widetilde{d_T}([\widetilde{\bm{x}},\widetilde{\bm{c}}]^T,[\widetilde{\bm{y}},\widetilde{\bm{d}}]^T).
	\end{align*}
	Hence, $G^\alpha$ is contractive with contraction rate $\rho(\bm{M})$.
\end{proof}

\section{Proof of \Cref{thm:local_optimum}}\label{sec:appendix_proof_local_convergence}

We start by recalling the well-known Euler theorem for homogeneous functions
\begin{equation}\label{eq:euler_thm}
f\in\text{hom}(1) \Leftrightarrow \bm{x}^T \nabla f(\bm{x}) = f(\bm{x}),
\end{equation}
which we will use for the functions $f(\bm{x})=\|\bm{x}\|_{p^\ast}$,  $f(\bm{c})=\|\bm{c}\|_{q^\ast}$, $f(\bm x) = f_\alpha(\bm x, \cdot)$, and $f(\bm c) = f_\alpha(\cdot, \bm c)$ in the following.
In particular, we will use the Euler theorem in the block vector form
\begin{equation}\label{eq:euler_multivariate}
\begin{bmatrix}
\bm{x}\\\bm{c}
\end{bmatrix}^T
\begin{bmatrix}
\nabla_{\bm{x}} f_\alpha(\bm{x},\bm{c})\\\nabla_{\bm{c}} f_\alpha(\bm{x},\bm{c})
\end{bmatrix}
= \bm{x}^T \nabla_{\bm{x}} f_\alpha(\bm{x},\bm{c}) + \bm{c}^T \nabla_{\bm{c}} f_\alpha(\bm{x},\bm{c})
=  f_\alpha(\bm{x},\bm{c}),
\end{equation}
since $f_\alpha(\bm{x},\cdot)\in\text{hom}(1)$ and $f_\alpha(\cdot,\bm{c})\in\text{hom}(1)$ implies $f_\alpha(\bm{x},\bm{c})\in\text{hom}(1)$. Next,  we recall the well-known H\"older inequality for norms
\begin{equation}\label{eq:hoelder_norms}
\langle \bm{x},\bm{y} \rangle \leq \|\bm{x}\|_p \|\bm{y}\|_{p^\ast},
\end{equation}
for $1/p+1/p^\ast=1$. Moreover, we recall that the definition of the dual of a homogeneous function $f_\alpha(\bm{x},\bm{c})$ via \cite{rubinov1998duality}
\begin{equation*}
f_\alpha^\ast(\bm{y},\bm{d}) = \sup_{[\bm{x},\bm{c}]^T\in\R_{>0}^{n+L}} \frac{\langle \begin{bmatrix}\bm{y}\\\bm{d}\end{bmatrix}, \begin{bmatrix}\bm{x}\\\bm{c}\end{bmatrix} \rangle}{f_\alpha(\bm{x},\bm{c})}
\end{equation*}
additionally gives rise to a H\"older-type inequality for homogeneous functions
\begin{equation}\label{eq:hoelder_hom_multivariate}
\langle \begin{bmatrix}
\bm{x}\\\bm{c}
\end{bmatrix}, 
\begin{bmatrix}
\bm{y}\\\bm{d}
\end{bmatrix}
\rangle
\leq f_\alpha(\bm{x},\bm{c}) f_\alpha^\ast(\bm{y},\bm{d}).
\end{equation}
Finally, we observe that from the definition of sub-differential \cite{fletcher2013practical}
\begin{equation*}
\partial f(\bm{x}) = \left\{ \bm{y} \colon \langle \bm{x}, \bm{y} \rangle = f(\bm{x})f^\ast(\bm{y})  \right\},
\end{equation*}
we obtain, for a differentiable and  homogeneous $f$, that
\begin{equation}\label{eq:dual_gradient_identity}
\langle x,\nabla f(x)\rangle = f(x)f^*(\nabla f(x))
\end{equation}
which, combined with \eqref{eq:euler_thm}, yields $f^*(\nabla f(x))=1$.

With these tools at hand, we obtain
\begin{align*}
g_\alpha(\bm{x}_k,\bm{c}_k) & = \frac{f_\alpha(\bm{x}_k,\bm{c}_k)}{ \|\bm{x}_k\|_p \|\bm{c}_k\|_q} \overset{\eqref{eq:euler_multivariate}}{=} \frac{\langle \bm{x}_k, \nabla_{\bm{x}} f_\alpha(\bm{x}_k,\bm{c}_k) \rangle + \langle \bm{c}_k, \nabla_{\bm{c}} f_\alpha(\bm{x}_k,\bm{c}_k) \rangle}{\|\bm{x}_k\|_p \|\bm{c}_k\|_q}\\
& = \frac{1}{\|\bm{c}_k\|_q} \langle \frac{\bm{x}_k}{\|\bm{x}_k\|_p}, \nabla_{\bm{x}} f_\alpha(\bm{x}_k,\bm{c}_k) \rangle + \frac{1}{\|\bm{x}_k\|_p} \langle \frac{\bm{c}_k}{\|\bm{c}_k\|_q}, \nabla_{\bm{c}} f_\alpha(\bm{x}_k,\bm{c}_k) \rangle\\
& \overset{\eqref{eq:hoelder_norms}}{\leq} \frac{1}{\|\bm{c}_k\|_q} \|\nabla_{\bm{x}} f_\alpha(\bm{x}_k,\bm{c}_k)\|_{p^\ast} + \frac{1}{\|\bm{x}_k\|_p} \|\nabla_{\bm{c}} f_\alpha(\bm{x}_k,\bm{c}_k)\|_{q^\ast}\\
& \overset{\eqref{eq:euler_thm}}{=} \frac{1}{\|\bm{c}_k\|_q} \langle \nabla_{\bm{x}} f_\alpha(\bm{x}_k,\bm{c}_k), J_{p^\ast}(\nabla_{\bm{x}} f_\alpha(\bm{x}_k,\bm{c}_k)) \rangle\\
& \qquad+ \frac{1}{\|\bm{x}_k\|_p} \langle \nabla_{\bm{c}} f_\alpha(\bm{x}_k,\bm{c}_k), J_{q^\ast}(\nabla_{\bm{c}} f_\alpha(\bm{x}_k,\bm{c}_k)) \rangle\\
& \overset{\eqref{eq:iteration_multilayer}}{=} \frac{1}{\|\bm{c}_k\|_q} \langle \nabla_{\bm{x}} f_\alpha(\bm{x}_k,\bm{c}_k), \bm{x}_{k+1} \rangle + \frac{1}{\|\bm{x}_k\|_p} \langle \nabla_{\bm{c}} f_\alpha(\bm{x}_k,\bm{c}_k), \bm{c}_{k+1} \rangle\\
& = \frac{\langle \begin{bmatrix}
	\bm{x}_{k+1}\\\bm{c}_{k+1}
	\end{bmatrix}, 
	\begin{bmatrix}
	\|\bm{x}_k\|_p \nabla_{\bm{x}} f_\alpha (\bm{x}_k,\bm{c}_k)\\
	\|\bm{c}_k\|_q \nabla_{\bm{c}} f_\alpha (\bm{x}_k,\bm{c}_k)
	\end{bmatrix}
	\rangle}{\|\bm{x}_k\|_p \|\bm{c}_k\|_q}\\
& \overset{\eqref{eq:hoelder_hom_multivariate}}{\leq} \frac{f_\alpha(\bm{x}_{k+1},\bm{c}_{k+1}) f_\alpha^\ast(\|\bm{x}_k\|_p \nabla_{\bm{x}} f_\alpha (\bm{x}_k,\bm{c}_k), \|\bm{c}_k\|_q \nabla_{\bm{c}} f_\alpha (\bm{x}_k,\bm{c}_k))}{\|\bm{x}_k\|_p \|\bm{c}_k\|_q}\\
& \overset{(\ast)}{=} \frac{f_\alpha(\bm{x}_{k+1},\bm{c}_{k+1}) \|\bm{x}_k\|_p \|\bm{c}_k\|_q f_\alpha^\ast(\nabla_{\bm{x}} f_\alpha (\bm{x}_k,\bm{c}_k), \nabla_{\bm{c}} f_\alpha (\bm{x}_k,\bm{c}_k))}{\|\bm{x}_k\|_p \|\bm{c}_k\|_q}\\
& = f_\alpha(\bm{x}_{k+1},\bm{c}_{k+1}) f_\alpha^\ast(\nabla_{\bm{x}} f_\alpha (\bm{x}_k,\bm{c}_k)), \nabla_{\bm{c}} f_\alpha (\bm{x}_k,\bm{c}_k))\\
& \overset{\eqref{eq:dual_gradient_identity}}{=} f_\alpha(\bm{x}_{k+1},\bm{c}_{k+1}) \overset{(\dag)}{=} \frac{f_\alpha(\bm{x}_{k+1},\bm{c}_{k+1})}{\|\bm{x}_{k+1}\|_p \|\bm{c}_{k+1}\|_q} = g_\alpha(\bm{x}_{k+1},\bm{c}_{k+1}),
\end{align*}
\begin{sloppypar}
	where $(\ast)$ uses the fact that from $f_\alpha(\bm{x},\cdot)\in\text{hom}(1)$ and $f_\alpha(\cdot,\bm{c})\in\text{hom}(1)$ it follows that $f_\alpha^\ast(\bm{x},\cdot)\in\text{hom}(1)$ and $f_\alpha^\ast(\cdot,\bm{c})\in\text{hom}(1)$ \cite{rubinov1998duality}.
	Furthermore, as $J_{p^*}$ and $J_{q^*}$ are the sub-differentials of the $p^*$ and $q^*$ norms, respectively, $(\dag)$ uses \mbox{$\|\bm{x}_{k+1}\|_p = \|J_{p^\ast} (\nabla_{\bm{x}} f_\alpha(\bm{x}_k,\bm{c}_k))\|_p = 1$} and \mbox{$\|\bm{c}_{k+1}\|_q = \|J_{q^\ast} (\nabla_{\bm{c}} f_\alpha(\bm{x}_k,\bm{c}_k))\|_q = 1$} as a consequence of \cref{eq:dual_gradient_identity}.
\end{sloppypar}

It remains to show that the two inequalities in the above derivation are strict unless the iteration reaches a critical point $[\bm{x}_k,\bm{c}_k]^T=[\bm{x}_{k+1},\bm{c}_{k+1}]^T$ at which $\nabla_{\bm{x}} g_\alpha(\bm{x}_k,\bm{c}_k) = \nabla_{\bm{c}} g_\alpha(\bm{x}_k,\bm{c}_k) = \bm{0}$.
The first inequality consists of the sum of two H\"older inequalities.
By the definition of the sub-differential, we have for all $\beta\in\R_{>0}$ that
\begin{align*}
\langle \frac{\bm{x}_k}{\|\bm{x}_k\|_p}, \nabla_{\bm{x}} f_\alpha(\bm{x}_k,\bm{c}_k) \rangle & = \|\nabla_{\bm{x}} f_\alpha(\bm{x}_k,\bm{c}_k)\|_{p^\ast} \Leftrightarrow \nabla_{\bm{x}} f_\alpha(\bm{x}_k,\bm{c}_k) = \beta J_p(\bm{x}_k)\\
& \Rightarrow J_{p^\ast}(\nabla_{\bm{x}} f_\alpha(\bm{x}_k,\bm{c}_k)) = \bm{x}_k,
\end{align*}
which yields $\bm{x}_{k+1}=\bm{x}_k$ by \cref{eq:iteration_multilayer} as well as $\nabla_{\bm{x}} g_\alpha(\bm{x}_k,\bm{c}_k)=\bm{0}$ since
\begin{equation*}
\nabla_{\bm{x}} g_\alpha(\bm{x}_k,\bm{c}_k)=\bm{0} \Leftrightarrow \nabla_{\bm{x}} f_\alpha(\bm{x}_k,\bm{c}_k) = \|\bm{c}_k\|_q g_\alpha(\bm{x}_k,\bm{c}_k) J_p(\bm{x}_k) \Leftrightarrow J_{p^\ast}(\nabla_{\bm{x}}f_\alpha(\bm{x}_k,\bm{c}_k)) = \bm{x}_k.
\end{equation*}
The statement
\begin{equation*}
\langle \frac{\bm{c}_k}{\|\bm{c}_k\|_q}, \nabla_{\bm{c}} f_\alpha(\bm{x}_k,\bm{c}_k) \rangle = \|\nabla_{\bm{c}} f_\alpha(\bm{x}_k,\bm{c}_k)\|_{q^\ast} \Rightarrow \bm{c}_{k+1}=\bm{c}_k \text{ and } \nabla_{\bm{c}} g_\alpha(\bm{x}_k,\bm{c}_k)=\bm{0}
\end{equation*}
is shown analogously. Thus, the first inequality is strict unless $[\bm{x}_k,\bm{c}_k]^T=[\bm{x}_{k+1},\bm{c}_{k+1}]^T$. In that case, we also have 
\[
\langle \begin{bmatrix}
\bm{x}_{k+1}\\\bm{c}_{k+1}
\end{bmatrix}, 
\begin{bmatrix}
\|\bm{x}_k\|_p \nabla_{\bm{x}} f_\alpha (\bm{x}_k,\bm{c}_k)\\
\|\bm{c}_k\|_q \nabla_{\bm{c}} f_\alpha (\bm{x}_k,\bm{c}_k)
\end{bmatrix}
\rangle
= 
f_\alpha(\bm{x}_{k+1},\bm{c}_{k+1}) f_\alpha^\ast(\|\bm{x}_k\|_p \nabla_{\bm{x}} f_\alpha (\bm{x}_k,\bm{c}_k), \|\bm{c}_k\|_q \nabla_{\bm{c}} f_\alpha (\bm{x}_k,\bm{c}_k)),
\] directly by \eqref{eq:dual_gradient_identity}, concluding the proof.

\section{Proof of \Cref{thm:QUBO_bounds}}\label{sec:appendix_proof_QUBO_bounds}

We start by rewriting \cref{eq:QUBO_elementwise} as
\begin{equation}\label{eq:QUBO_present_vs_missing}
\sum_{k=1}^L \frac{\bm{c}_k}{\|\bm{c}\|_1} \sum_{i,j=1}^n \left( \frac{\bm{A}_{ij}^{(k)}}{n_1^{(k)}} - \frac{1 - \bm{A}_{ij}^{(k)}}{n_2^{(k)}} \right) \max \{ \bar{\bm{x}}_i, \bar{\bm{x}}_j\}
\end{equation}
For $L=1$, i.e., $\bm{c}=1$, let $s$ denote the core size represented by $\bar{\bm{x}}$ after row and column permutations in $\bm{A}^{(1)}$, i.e., $\bar{\bm{x}}$ is $1$ in the first $s$ entries and $0$ otherwise.
Then, $\max \{\bm{x}_i, \bm{x}_j \} = 1$ if $i<s$ or $j<s$, i.e., all entries belonging to the L-shape of size $s$ are counted in \cref{eq:QUBO_present_vs_missing}.
By definition, we have $\sum_{i,j=1}^n \bm{A}_{ij}^{(1)} = n_1^{(1)}$ and $\sum_{i,j=1}^n (1 - \bm{A}_{ij}^{(1)}) = n_2^{(1)}$.
This means that the value of $1$ in \cref{eq:QUBO_present_vs_missing} can be obtained if and only if all $n_1^{(1)}$ non-zeros of $\bm{A}^{(1)}$ lie within the L-shape of ideal size $s^\ast$ and contribute positively to the sum while no missing edges must be present in the L-shape of size $s^\ast$ that would contribute negatively.
The analogous argumentation holds for the complementary network of the ideal L-shape of size $s^\ast$ and yields the lower bound of $-1$.

For $L > 1$, the above argument holds for every term
\begin{equation*}
\sum_{i,j=1}^n \left( \frac{\bm{A}_{ij}^{(k)}}{n_1^{(k)}} - \frac{1 - \bm{A}_{ij}^{(k)}}{n_2^{(k)}} \right) \max \{ \bar{\bm{x}}_i, \bar{\bm{x}}_j\}
\end{equation*}
for all $k=1, \dots , L$.
The index $s^\ast$ must be the same in all layers since $\bar{\bm{x}}$ is the same for all layers.
The bounds \cref{eq:QUBO_bounds} follow for general $L\in\N$ since $\sum_{k=1}^L \frac{\bm{c}_k}{\|\bm{c}\|_1} = 1$.

Finally, the vector $\bar{\bm{x}}=\bm{1}$ leads to $\max \{ \bar{\bm{x}}_i, \bar{\bm{x}}_j\} = 1$ for all $i,j=1,\dots,n$ and hence, by definition of $n_1^{(k)}$ and $n_2^{(k)}$,
\begin{equation*}
\sum_{i,j=1}^n \left( \frac{\bm{A}_{ij}^{(k)}}{n_1^{(k)}} - \frac{1 - \bm{A}_{ij}^{(k)}}{n_2^{(k)}} \right) \max \{ \bar{\bm{x}}_i, \bar{\bm{x}}_j\} = \frac{\sum_{i,j=1}^n \bm{A}_{ij}^{(k)}}{n_1^{(k)}} - \frac{\sum_{i,j=1}^n \left( 1 - \bm{A}_{ij}^{(k)} \right)}{n_2^{(k)}} = 1-1 = 0,
\end{equation*}
for all $k=1, \dots , L$.
Since, in practice, we take the maximum multiplex QUBO values over the binary vectors $\bar{\bm{x}}$, we have a lower bound of $0$.

\enlargethispage{20pt}

\dataccess{All data used in the numerical experiments is publicly available as indicated in \Cref{sec:numerical_experiments}.
All codes for reproducing the results presented in this paper are publicly available under \url{https://github.com/COMPiLELab/MPNSM}.}

\aiuse{We have not used AI-assisted technologies in creating this article.}

\aucontribute{K.B.: Conceptualization; Data curation; Formal analysis; Funding acquisition; Investigation; Software; Visualization; Writing – original draft; M.S.: Conceptualization; Investigation; Supervision; Validation; Writing – review \& editing F.T.: Conceptualization; Formal analysis; Investigation; Methodology; Supervision; Validation; Visualization; Writing – review \& editing}

\competing{We declare we have no competing interests.}

\funding{This work was supported by a fellowship of the German Academic Exchange Service (DAAD).}

\ack{KB thanks FT and the Numerical Analysis group at GSSI for their hospitality during his research stay in L'Aquila as well as the German Academic Exchange Service (DAAD) for its funding.}


\vskip2pc


%
%
%
%
%
%
%
%
%

\end{document}